%% file: MASA_FOn.tex
\documentclass[a4paper,11pt]{amsart}

\usepackage[utf8]{inputenc}
\usepackage[T1]{fontenc}
\usepackage{amsmath, amssymb, amsfonts,stmaryrd}
\usepackage[a4paper, hmargin=1.5cm, vmargin=2.5cm]{geometry}
\usepackage{hyperref}

\title{The radial MASA in free orthogonal quantum groups}
\author{Amaury Freslon}
\address{Amaury Freslon, Laboratoire de Math\'ematiques d'Orsay, Univ. Paris-Sud, CNRS, Universit\'e Paris-Saclay, 91405 Orsay, France}
\email{amaury.freslon@math.u-psud.fr}
\author{Roland Vergnioux}
\address{Roland Vergnioux, Normandie Univ, UNICAEN, CNRS, Laboratoire de Mathématiques Nicolas Oresme, 14000 Caen, France}
\email{roland.vergnioux@unicaen.fr}
 % General theory of von Neumann algebras
\date{}

\theoremstyle{plain}
\newtheorem{thm}{Theorem}[section]
\newtheorem{prop}[thm]{Proposition}
\newtheorem{cor}[thm]{Corollary}
\newtheorem{lem}[thm]{Lemma}

\theoremstyle{definition}
\newtheorem{de}[thm]{Definition}

\DeclareMathOperator{\id}{id}
\DeclareMathOperator{\Ir}{Irr}
\DeclareMathOperator{\Mor}{Hom}
\DeclareMathOperator{\Pol}{Pol}
\DeclareMathOperator{\Span}{span}
\DeclareMathOperator{\Tr}{Tr}
\DeclareMathOperator{\tr}{tr}

\newcommand{\B}{\mathcal{B}}
\newcommand{\Ll}{\mathcal{L}}
\newcommand{\C}{\mathbb{C}}
\newcommand{\E}{\mathbb{E}}
\newcommand{\F}{\mathbb{F}}
\newcommand{\G}{\mathbb{G}}
\newcommand{\N}{\mathbb{N}}
\newcommand{\Z}{\mathbb{Z}}
\newcommand{\HS}{{\rm HS}}
\newcommand{\red}{{\rm red}}

\newcommand{\color}[2][]{}

\begin{document}

\begin{abstract}
  We prove that the radial subalgebra in free orthogonal quantum group factors
  is maximal abelian and mixing, and we compute the associated bimodule. The
  proof relies on new properties of the Jones-Wenzl projections and on an
  estimate of certain scalar products of coefficients of irreducible
  representations.
\end{abstract}

\maketitle

\section{Introduction}

Discrete groups have been an important part of the theory of von Neumann
algebras since its very beginning. Taking advantage of their algebraic or
geometric properties, one can build interesting families of examples and
counter-examples of von Neumann algebras, and get some insight into crucial
structural properties like property (T) or approximation properties. In the last
ten years, there has been an increasing number of results showing that discrete
quantum groups can also produce interesting examples of von Neumann algebras. In
this work, we continue this program by initiating the study of abelian
subalgebras in von Neumann algebras of discrete quantum groups.

The importance of abelian subalgebras in the study of von Neumann algebras has
been long known and, as already mentioned, group von Neumann algebras have
played an important role in that history. For instance, the subalgebra generated
by one of the generating copies of $\Z$ inside the von Neumann algebra of the
free group $\F_{2}$ was proved by J.~Dixmier to be maximal abelian
\cite{dixmier1954abeliens}, and by S.~Popa to be maximal injective
\cite{popa1983maximal}, thus answering a long-standing question of
R.V.~Kadison. The fact that the subalgebra comes from a group inclusion was
crucial there.

Another example of abelian subalgebra in free group factors is the so-called
\emph{radial} (or \emph{laplacian}) subalgebra, which is the one generated by
the sum of the generators and their inverses. This subalgebra does not come from
a subgroup, hence the aforementioned techniques do not apply. S.~Radulescu
introduced in \cite{radulescu1991singularity} tools to prove that this
subalgebra, which was already known to be maximal abelian by work of S.~Pytlik
\cite{pytlik1981radial}, is singular. His techniques were later used again to
prove that the radial subalgebra is maximal amenable
\cite{cameron2010radial}. For more background on maximal abelian subalgebras we
refer to the book \cite{sinclair2008finite}.

\bigskip

In this paper, we study the analogue of the radial subalgebra in free quantum
group factors. More precisely, we consider the free orthogonal quantum group of
Kac type $O_{N}^{+}$ and, inside its von Neumann algebra $L^\infty(O_{N}^{+})$,
the subalgebra generated by the characters of irreducible
representations. Recall that $O_{N}^{+}$ is a compact quantum group introduced
in \cite{wang1995free}, whose discrete dual is a quantum analogue of a free
group. In particular $L^\infty(O_{N}^{+})$ plays the role of a free group factor
$\Ll(\F_{N})$. This analogy, dating back to the seminal works of T.~Banica
\cite{banica1996theorie}, \cite{banica1997groupe}, has been supported since then
by further work of several authors who proved that the von Neumann algebra
$L^{\infty}(O_{N}^{+})$, for $N\geqslant 3$, indeed shares many properties with
free group factors:
\begin{itemize}
\item it is a full factor with the Akemann-Ostrand property
  \cite{vaes2007boundary},
\item it has the Haagerup property \cite{brannan2011approximation} and the
  completely contractive approximation property \cite{freslon2012examples},
\item it is strongly solid \cite{isono2012examples} and has property strong $HH$
  \cite{fima2014cocycle},
\item it satisfies the Connes embedding conjecture \cite{brannan2014connes}.
\end{itemize}

\bigskip

As far as the radial subalgebra is concerned, the techniques of S.~Radulescu do
not apply in the quantum case, because there is no clear way to mimic the
construction of the so-called \emph{Radulescu basis}. However, the properties of
the radial algebra mentioned previously can all be proved using another tool
which we briefly explain. Consider, for $l\in \N$, the element $w_{l}\in
\Ll(\F_{N})$ which is the sum of all words of length $l$. Then, if $x, x'$ are
two words of length $k$ and $y, y'$ are two other words of length $n$, we have
\begin{equation} \label{eq:classicalestimate} \langle (x-x')w_{l},
  w_{l'}(y-y')\rangle \leqslant 2\min(k+1,n+1).
\end{equation}
This estimate can be proved by elementary counting arguments, similar to the
ones in \cite[Sec 4]{sinclair2003laplacian}. It can then be used to prove
maximal abelianness and singularity in one shot. We will use the same strategy
here.

Elements of the form $x-x'$ with $x$ and $x'$ of the same length $k$ form a
basis of the orthogonal of the radial subalgebra in $\Ll(\F_{N})$, in the
quantum case their role will be played by the coefficients $u^{k}_{\xi\eta}$ of
an irreducible representation $u^{k}$ with respect to vectors $\xi$, $\eta$ such
that $\xi$ is orthogonal to $\eta$. The role of $w_{l}$ will be played by the
character $\chi_{l}$ of the irreducible representation $u^{l}$ --- note however
that $\|w_{l}\|^2 = 2N(2N-1)^{l-1}$ in $L^2(\F_{N})$, whereas $\|\chi_{l}\|^{2}
= 1$ in $L^{2}(O_{N}^{+})$. The estimate analogous to
\eqref{eq:classicalestimate} that we will prove and use in the present article
is then stated as follows (see Theorem~\ref{thm:keyestimate}):
\begin{equation*}
  \langle \chi_{l}u_{\xi', \eta'}^{k}, u^{n}_{\xi, \eta}\chi_{l'} \rangle \leqslant Kq^{\max(l, l')},
\end{equation*}
with $q \in ]0,1[$. From this we will deduce all the results announced in the
abstract.

\bigskip

Let us now outline the content of the paper. In Section~\ref{sec:preliminaries},
we recall some facts on compact quantum groups and in particular on free
orthogonal quantum groups. Since the geometry of their representation theory
will be crucial in the computations, we have to make some conventional choices
and give the corresponding explicit formul\ae\ for several related objects.

Section~\ref{sec:Wenzl} and~\ref{sec:estimate} form the core of the paper. There
we prove the announced estimate for scalar products of coefficients and
characters. The proof, presented in Section~\ref{sec:estimate}, is quite
technical and relies on properties of the so-called Jones-Wenzl projections
which are of independent interest and are established in
Section~\ref{sec:Wenzl}.

Eventually, we prove in Section~\ref{sec:radial} all our structural results on
the radial subalgebra, namely that it is maximal abelian, mixing and has
spectral measure equivalent to the Lebesgue measure. The proofs here are very
simple using the main estimate and the arguments are certainly well-known to
experts in von Neumann algebras. Since however people interested in discrete
quantum groups may not be so familiar with them, we give full proofs. The paper
ends with some remarks on the results of this work.

\subsection*{Acknowledgments}

We would like to thank \'Eric Ricard and Cyril Houdayer for interesting
conversations. The first author was partially supported by the ERC Advanced
Grant 339760 "Noncommutative distributions in free probability". 
% ERC-2013-ADG 339760

\section{Preliminaries}\label{sec:preliminaries}

In this section we give the basic definitions and results needed in the
paper. All scalar products will be \emph{left-linear} and we will denote by
$\B(H)$ the algebra of all bounded operators on a Hilbert space $H$. When
considering an operator $X\in \B(H_{1}\otimes H_{2})$, we will use the
\emph{leg-numbering notations},
\begin{equation*}
  X_{12} := X\otimes1, X_{23} := 1\otimes X \text{ and } X_{13} := (\Sigma\otimes 1)(1\otimes X)(\Sigma\otimes 1),
\end{equation*}
where $\Sigma : H_{1}\otimes H_{2} \rightarrow H_{2}\otimes H_{1}$ is the flip
map. For any two vectors $\xi, \eta\in H$, we define a linear form $\omega_{\xi
  \eta} : \B(H)\rightarrow \C$ by $\omega_{\xi \eta}(T) = \langle T(\xi),
\eta\rangle$.

\subsection{Compact quantum groups}

We briefly review the theory of compact quantum groups as introduced by
S.L.~Woronowicz in \cite{woronowicz1995compact}. In the sequel, all tensor
products of C*-algebras are spatial and we denote by $\overline\otimes$ the
tensor product of von Neumann algebras.

\begin{de}
  A \emph{compact quantum group} $\G$ is a pair $(C(\G), \Delta)$ where $C(\G)$
  is a unital C*-algebra and $\Delta : C(\G)\rightarrow C(\G)\otimes C(\G)$ is a
  unital $*$-homomorphism such that
  \begin{equation*}
    (\Delta\otimes \id)\circ\Delta = (\id\otimes\Delta)\circ\Delta
  \end{equation*}
  and the spaces $\Span\{\Delta(C(\G))(1\otimes C(\G))\}$ and
  $\Span\{\Delta(C(\G))(C(\G)\otimes 1)\}$ are both dense in $C(\G)\otimes
  C(\G)$.
\end{de}

According to \cite[Thm 1.3]{woronowicz1995compact}, any compact quantum group
$\G$ has a unique \emph{Haar state} $h\in C(\G)^{*}$, satisfying
\begin{eqnarray*}
  (\id\otimes h)\circ \Delta(a) = h(a).1 \\
  (h\otimes \id)\circ \Delta(a) = h(a).1
\end{eqnarray*}
for all $a\in C(\G)$. Let $(L^{2}(\G), \pi_{h}, \Omega)$ be the associated GNS
construction and let $C_\red(\G)$ be the image of $C(\G)$ under the GNS
representation $\pi_{h}$. It is called the \emph{reduced C*-algebra} of $\G$ and
its bicommutant in $\B(L^{2}(\G))$ is the \emph{von Neumann algebra of $\G$},
denoted by $L^{\infty}(\G)$. To study this object, we will use representations
of compact quantum groups.

\begin{de}
  A \emph{representation} of a compact quantum group $\G$ on a Hilbert space $H$
  is an operator $u\in L^{\infty}(\G)\overline{\otimes} \B(H)$ such that
  $(\Delta\otimes \id)(u) = u_{13}u_{23}$. It is said to be \emph{unitary} if
  the operator $u$ is unitary.
\end{de}

\begin{de}
  Let $\G$ be a compact quantum group and let $u$ and $v$ be two representations
  of $\G$ on Hilbert spaces $H_{u}$ and $H_{v}$ respectively. An
  \emph{intertwiner} (or \emph{morphism}) between $u$ and $v$ is a map
  $T\in\B(H_{u}, H_{v})$ such that $v(\id\otimes T) = (\id\otimes T)u$. The set
  of intertwiners between $u$ and $v$ will be denoted by $\Mor(u, v)$.
\end{de}

A representation $u$ is said to be \emph{irreducible} if $\Mor(u, u) = \C.\id$
and it is said to be \emph{contained} in $v$ if there is an injective
intertwiner between $u$ and $v$. We will say that two representations are
\emph{equivalent} (resp. \emph{unitarily equivalent}) if there is an intertwiner
between them which is an isomorphism (resp. a unitary). Let us define two
fundamental operations on representations.

\begin{de}
  Let $\G$ be a compact quantum group and let $u$ and $v$ be two representations
  of $\G$ on Hilbert spaces $H_{u}$ and $H_{v}$ respectively. The \emph{direct
    sum} of $u$ and $v$ is the diagonal sum of the operators $u$ and $v$ seen as
  an element of $L^{\infty}(\G)\overline{\otimes} \B(H_{u}\oplus H_{v})$. It is
  a representation denoted by $u\oplus v$. The \emph{tensor product} of $u$ and
  $v$ is the element $u_{12}v_{13}\in L^{\infty}(\G)\overline{\otimes}
  \B(H_{u}\otimes H_{v})$. It is a representation denoted by $u\otimes v$.
\end{de}

The theory of representations of compact groups can be generalized to this
setting (see \cite[Section 6]{woronowicz1995compact}). If $u$ is a
representation of $\G$ on a Hilbert space $H$ and if $\xi, \eta\in H$, then
$u_{\xi \eta} = (\id\otimes \omega_{\xi \eta})(u)\in C(\G)$ is called a
\emph{coefficient} of $u$.

\begin{thm}[Woronowicz]
  Every representation of a compact quantum group is equivalent to a unitary
  one. Every irreducible representation of a compact quantum group is
  finite-dimensional and every unitary representation is unitarily equivalent to
  a sum of irreducible ones. Moreover, the linear span of the coefficients of
  all irreducible representations is a dense Hopf $*$-subalgebra of $C(\G)$
  denoted by $\Pol(\G)$.
\end{thm}

\subsection{Irreducible representations}\label{subsec:irreducible}

Let $\Ir(\G)$ be the set of equivalence classes of irreducible unitary
representations of $\G$. For $\alpha\in \Ir(\G)$, we will denote by $u^{\alpha}$
a representative of the class $\alpha$ and by $H_{\alpha}$ the
finite-dimensional Hilbert space on which $u^{\alpha}$ acts. The scalar product
induced by the Haar state can be easily computed on coefficients of irreducible
representations by \cite[Eq. 6.7]{woronowicz1995compact}:
\begin{equation*}
  \left\langle u^{\alpha}_{\xi \eta}, u^{\beta}_{\xi' \eta'}\right\rangle = \delta_{\alpha, \beta}\frac{\langle \xi, \xi'\rangle \langle \eta', Q_{\alpha}\eta \rangle}{d_{\alpha}}
\end{equation*}
where $Q_{\alpha}$ is a positive matrix determined by the representation
$\alpha$ and $d_{\alpha} = \Tr(Q_\alpha) = \Tr(Q_\alpha^{-1})>0$ is called the
\emph{quantum dimension} of $\alpha$. Note that in general, $d_{\alpha}$ is
greater than $\dim(H_{\alpha})$. However, it is easy to see that the two
dimensions agree if and only if $Q_{\alpha} = \id$. When this is the case for
all $\alpha \in\Ir(\G)$ we say that $\G$ is of \emph{Kac type}.

Because the coefficients of irreducible representations are dense in $C(\G)$, it
is enough to understand products of those coefficients to describe the whole
C*-algebra structure of $C(\G)$. For simplicity, we will assume from now on that
for any two irreducible representations $\alpha$ and $\beta$, every irreducible
subrepresentation of $\alpha\otimes\beta$ appears with multiplicity one (this
assumption will always be satisfied when considering free orthogonal quantum
groups). For such a subrepresentation $\gamma$ of $\alpha\otimes\beta$, let
$v^{\alpha, \beta}_{\gamma}$ be an isometric intertwiner from $H_{\gamma}$ to
$H_{\alpha}\otimes H_{\beta}$. Then,
\begin{equation}\label{eq:product}
  u^{\alpha}_{\xi \eta}u^{\beta}_{\xi' \eta'} = \sum_{\gamma\subset \alpha\otimes \beta} u^{\gamma}_{(v_{\gamma}^{\alpha, \beta})^{*}(\xi\otimes \xi'), (v_{\gamma}^{\alpha, \beta})^{*}(\eta\otimes \eta')}.
\end{equation}
Note that even though $v^{\alpha, \beta}_{\gamma}$ is only defined up to a
complex number of modulus one, the sesquilinearity of the scalar product ensures
that the expression above is independent of this phase. We will also use the
projection $P^{\alpha,\beta}_\gamma\in \B(H_\alpha\otimes H_\beta)$ onto the
$\gamma$-homogeneous component, $P^{\alpha,\beta}_\gamma = v^{\alpha,
  \beta}_{\gamma}v^{\alpha, \beta *}_{\gamma}$, which is again independent of
the choice of $v^{\alpha, \beta}_{\gamma}$.

For any $\alpha\in \Ir(\G)$, there is a unique (up to unitary equivalence)
irreducible representation, called the \emph{contragredient representation} of
$\alpha$ and denoted by $\overline{\alpha}$, such that $\Mor(\varepsilon,
\alpha\otimes \overline{\alpha})\neq \{0\} \neq \Mor(\varepsilon,
\overline{\alpha}\otimes \alpha)$, $\varepsilon$ denoting the trivial
representation (i.e. the element $1\otimes 1\in L^{\infty}(\G)\otimes \C$). We
choose morphisms $t_{\alpha} \in \Mor(\varepsilon, \alpha\otimes
\overline{\alpha})$ and $s_{\alpha} \in \Mor(\varepsilon,
\overline{\alpha}\otimes\alpha)$ connected by the conjugate equation
\begin{equation*}
  (\id_{\alpha}\otimes s_{\alpha}^{*})(t_{\alpha}\otimes\id_{\alpha}) = \id_{\alpha},
\end{equation*}
and normalized so that $\|s_{\alpha}\| = \|t_{\alpha}\| =
\sqrt{d_{\alpha}}$. Then, $t_{\alpha}$ is unique up to a phase and $s_{\alpha}$
is determined by $t_{\alpha}$.  The morphism $t_{\alpha}$ induces a
conjugate-linear isomorphism $j_{\alpha} : H_{\alpha}\rightarrow
H_{\overline{\alpha}}$ such that, setting $j_{\alpha}(\xi) = \overline\xi$,
\begin{equation*}
  t_{\alpha} = \sum_{i=1}^{\dim(H_{\alpha})} e_{i} \otimes \overline{e_{i}}
\end{equation*}
for any orthonormal basis $(e_{i})_{i}$ of $H_{\alpha}$. Note that $j_{\alpha}$
need not be a multiple of a conjugate-linear isometry in general --- this is
however the case if $\G$ is of Kac type. Let us also record the general fact
that the map $\overline{v}^{\alpha,\beta}_{\gamma} : H_{\overline\gamma} \to
H_{\overline\beta}\otimes H_{\overline\alpha}$ defined by
\begin{equation*}
  \overline{\xi} \mapsto  \Sigma(v^{\alpha, \beta}_{\gamma}(\xi))^{\bar~\otimes\bar~}
\end{equation*}
is an isometric morphism from $\overline\gamma$ to
$\overline\beta\otimes\overline\alpha$. In particular, when there is no
multiplicity in the fusion rules $\overline{v}^{\alpha,\beta}_{\gamma}$
coincides with $v^{\overline\beta,\overline\alpha}_{\overline\gamma}$ up to a
complex number of modulus one.

\subsection{Free orthogonal quantum groups}\label{subsec:freeorthogonal}

We will be concerned in the sequel with the free orthogonal quantum groups
introduced by S.~Wang and A.~van~Daele in \cite{wang1995free} and
\cite{van1996universal}. This subsection is devoted to briefly recalling their
definition and main properties.

\begin{de}\label{de:freeqgroups}
  For $N\in \N$, we denote by $C(O_{N}^{+})$ the universal unital C*-algebra
  generated by $N^{2}$ \emph{self-adjoint} elements $(u_{ij})_{1\leqslant i,
    j\leqslant N}$ such that the matrix $u=(u_{ij})$ is \emph{unitary}. For $Q
  \in GL_N(\C)$, we denote by $C(O^+(Q))$ the unital C*-algebra generated by
  $N^{2}$ elements $(u_{ij})_{1\leqslant i, j\leqslant N}$ such that the matrix
  $u=(u_{ij})$ is \emph{unitary} and $Q\overline{u}Q^{-1} = u$, where $\overline
  u = (u_{ij}^*)$.
\end{de}

One can check that there is a unique $*$-homomorphism $\Delta : C(O^{+}(Q))
\rightarrow C(O^{+}(Q))\otimes C(O^{+}(Q))$ such that for all $i, j$,
\begin{equation*}
  \Delta(u_{ij}) = \sum_{i, j = 0}^{N}u_{ik}\otimes u_{kj}.
\end{equation*}

\begin{de}
  The pair $O_{N}^{+} = (C(O_{N}^{+}), \Delta)$ is called the \emph{free
    orthogonal quantum group} of size $N$. The pair $O^{+}(Q) = (C(O^{+}(Q)),
  \Delta)$ is called the free orthogonal quantum group of parameter $Q$.
\end{de}

One can show that the compact quantum group $O^{+}(Q)$ is of Kac type if and
only if $Q$ is a scalar multiple of a unitary matrix. Although all results of
this article apply to general free orthogonal quantum groups of Kac type with
$N\geqslant 3$, we will restrict for simplicity to the case of $O_N^+$ --- see
Section~\ref{sec:radial} for comments about the non-Kac type.  The
representation theory of free orthogonal quantum groups was computed by
T.~Banica in~\cite{banica1996theorie}:

\begin{thm}[Banica]\label{thm:freefusion}
  The equivalence classes of irreducible representations of $O_{N}^{+}$ are
  indexed by the set of integers ($u^{0}$ being the trivial representation and
  $u^{1} = u$ the fundamental one), each one is isomorphic to its contragredient
  and the tensor product is given inductively by
  \begin{equation*}
    u^{1}\otimes u^{n} = u^{n+1}\oplus u^{n-1}.
  \end{equation*}
  If $N=2$, then $d_{n} = n+1$. Otherwise,
  \begin{equation*}
    d_{n} = \frac{q^{n+1} - q^{-n-1}}{q - q^{-1}},
  \end{equation*}
  where $q + q^{-1}= N$ and $0< q< 1$. Moreover, $O_{N}^{+}$ is
  of Kac type, hence $d_{n} = \dim(H_{n})$.
\end{thm}
There is an elementary estimate on $d_{n}$ given by $q^{-n}(1-q^{2})\leqslant
d_{n}\leqslant q^{-n}/(1-q^{2})$ . We will use it several times in the sequel
without referring to it explicitly.

\bigskip

To be able to do computations, we will use a particular set of representatives
of the irreducible representations. More precisely, let $H_{1} = \C^{N}$ be the
carrier space of the fundamental representation $u = u^{1}$. Then, for each
$n\in \N$, we let $H_{n}$ be the unique subspace of $H_{1}^{\otimes n}$ on which
the restriction of $u^{\otimes n}$ is equivalent to $u^{n}$. We denote by
$\id_{n}$ the identity of $H_{n}$.

It is easy to check that the map $t_{1} = \sum_{i=1}^{N} e_{i}\otimes e_{i}$
satisfies the requirements for the distinguished morphism $t_{u} \in
\Mor(\varepsilon, u\otimes \overline{u})$ as defined in the previous subsection,
with $\bar u= u$ and $s_1=t_1$. We fix this choice in the rest of the article
and we set
\begin{equation*}
  t_{n} = (P_{n}\otimes P_{n}) (t_{1})_{1, 2n}(t_{1})_{2, 2n-1} \dots (t_{1})_{n, n+1} \in H_{n}\otimes H_{n}.
\end{equation*}
We then have $s_{n} = t_{n}$, $j_{n}\circ j_{n} = \id_{n}$, and $j_{n}$ is a
conjugate linear unitary. The standard trace on $\B(H_{n})$ is given by
\begin{equation*}
  \Tr_{n}(f) = t_{n}^{*}(f\otimes \id)t_{n}
\end{equation*}
and the normalized trace by $\tr_{n}(f) = d_{n}^{-1}\Tr_{n}(f)$. Moreover,
writing again $\overline\zeta = j_{n}(\zeta)$ for $\zeta\in H_{n}$ we have
\begin{equation*}
  t_{n}^{*}(\zeta\otimes\id_{n}) = \overline{\zeta}^{*} \text{ and } t_{n}^{*}(\id_{n}\otimes\zeta) = s_{n}^{*}(\id_{n}\otimes\zeta) = \overline{\zeta}^{*}.
\end{equation*}

We will denote by $P_{n}$ the orthogonal projection from $H_{1}^{\otimes n}$
onto $H_{n}$, sometimes called the \emph{Jones-Wenzl projection}. Note that if
$a+b = n$, then $P_{n}(P_{a}\otimes P_{b}) = P_{n}$, so that we may also see
$P_{n}$ as an element of $\B(H_{a}\otimes H_{b})$. In other words we have, with
the notation of the previous subsection, $P_n = P_n^{a,b}$ for any $a$, $b$ such
that $a+b=n$.  The sequence of projections $(P_{n})_{n\in \N}$ satisfies the
so-called \emph{Wenzl recursion relation} (see for instance \cite[Eq
3.8]{frenkel1997canonical} or \cite[Eq 7.4]{vaes2007boundary}):
\begin{equation}\label{eq:wenzlrecursion}
  P_{n} = (P_{n-1}\otimes \id_{1}) + \sum_{l=1}^{n-1} (-1)^{n-l} \frac{d_{l-1}}{d_{n-1}} \left(\id_{1}^{\otimes (l-1)}\otimes t_{1}\otimes \id_{1}^{\otimes (n-l-1)}\otimes t_{1}^{*}\right)(P_{n-1}\otimes \id_{1}).
\end{equation}
We also record the following obvious fact, which will be used frequently in the
sequel without explicit reference: for any $a$, $b$ we have $(\id_{a}\otimes
t_{1}\otimes \id_{b})^{*}P_{a+b+2} = 0$. Indeed the image of $(\id_{a}\otimes
t_{1}\otimes\id_{b})^{*}$ is contained in $H_{a}\otimes H_{b}$ which has no
component equivalent to $H_{a+b+2}$. A first application is the following
reduced form of the Wenzl relation above, which is actually the original
relation presented in \cite{wenzl1987projections}:
\begin{equation}\label{eq:wenzlrecursionreduced}
  P_{n} = (P_{n-1}\otimes \id_{1}) - \frac{d_{n-2}}{d_{n-1}} (P_{n-1}\otimes \id_{1}) \left(\id_{1}^{\otimes (n-2)}\otimes t_{1} t_{1}^{*}\right)(P_{n-1}\otimes \id_{1}).
\end{equation}
We also have a reflected version as follows:
\begin{equation}\label{eq:wenzlrecursionreflected}
  P_{n} = (\id_{1}\otimes P_{n-1}) - \frac{d_{n-2}}{d_{n-1}}(\id_{1}\otimes P_{n-1}) \left(t_{1} t_{1}^{*}\otimes \id_{1}^{\otimes (n-2)}\right)(\id_{1}\otimes P_{n-1}).
\end{equation}

\section{Manipulating the Jones-Wenzl projections}
\label{sec:Wenzl}
 
In this section we establish two results concerning the sequence of projections
$P_{n}$ in the representation category of $O_N^+$. The first one studies partial
traces of these projections, while the second one is a kind of generalization of
Wenzl's recursion relation.

\subsection{Partial traces of projections}

The first result we need concerns projections onto irreducible representations
that are cut down by a trace. To explain what is going on, let us first consider
two integers $a, b\in \N$. Then, the operator
\begin{equation*}
  x_{a, b} = (\id_{a}\otimes\tr_{b})(P_{a+b}) =  d_{b}^{-1}(\id_{a}\otimes t_{b}^{*})(P_{a+b}\otimes \id_{b})(\id_{a}\otimes t_{b})\in \B(H_{\alpha})
\end{equation*}
is a scalar multiple of the identity because it is an intertwiner and $u^{a}$ is
irreducible. Of course, the same holds for $(\tr_{b}\otimes\id_{c})(P_{b+c})\in
\B(H_{c})$. However in general $x_{a, b, c} =
(\id_{a}\otimes\tr_{b}\otimes\id_{c})(P_{a+b+c})$ is not a scalar multiple of
the identity. In fact, an easy explicit computation already shows that $x_{1, 1,
  1} \in \B(H_{1}\otimes H_{1})$ is a non-trivial linear combination of the
identity and the flip map, in particular it is not even an
intertwiner. Proposition \ref{prop:partialtrace}, which is the main result of
this subsection, shows that when $b$ tends to $+\infty$, the partially traced
projection $x_{a, b, c}$ becomes asymptotically scalar.

To prove this, we need a lemma concerning the following construction: for a
linear map $f\in \B(H_{k})$, we define its \emph{rotated version} $\rho(f)$ by
\begin{equation*}
  \rho(f) = (P_{k}\otimes t_{1}^{*})(\id_{1}\otimes f\otimes\id_{1})(t_{1}\otimes P_{k})\in \B(H_{k}).
\end{equation*}
Diagrammatically, this transformation is represented as follows:
\begin{equation*}
  \input{rotation_def.latex}
\end{equation*}
In the sequel, $\|.\|_{\HS}$ will denote the non-normalized Hilbert-Schmidt
norm, i.e. $\|f\|_{HS}^{2} = \Tr(f^{*}f)$.

\begin{lem}\label{lem:tracezero}
  For any $f\in \B(H_{k})$, $\Tr(\rho(f)) = (-1)^{k-1}\Tr(f)/d_{k-1}$. Moreover,
  we have $\|\rho(f)\|_{\HS} \leqslant \|f\|_{\HS}$.
\end{lem}

\begin{proof}
  For $k=1$ we have
  \begin{align*}
    \Tr(\rho(f)) & = t_{1}^{*}(\id_{1}^{\otimes 2}\otimes t_{1}^{*}) (\id_{1}^{\otimes 2}\otimes f\otimes\id_{1})(\id_{1}\otimes t_{1}\otimes \id_{1})t_{1} \\
    & = t_{1}^{*}(f\otimes \id_{1})(t_{1}^{*}\otimes\id_{1}^{\otimes 2})(\id_{1}\otimes t_{1}\otimes \id_{1})t_{1} \\
    & = t_{1}^{*}(f\otimes\id_{1})t_{1} = \Tr(f).
  \end{align*}
  On diagrams, computing the trace corresponds to connecting upper and lower
  points pairwise by non-crossing lines on the left or on the
  right. Representing this by dotted lines for clarity, the computation above
  can be pictured as follows:
  \begin{equation*}
    \input{rotation_trace_init.latex}
  \end{equation*}
  When $k\geqslant 2$, we first perform the transformation
  \begin{align*}
    \Tr(\rho(f)) & = \Tr((P_{k}\otimes t_{1}^{*})(\id_{1}\otimes f\otimes\id_{1})(t_{1}\otimes \id_{1}^{\otimes k})) \\
    & = \Tr((\id_1^{\otimes k-1}\otimes
    t_{1}^{*})(P_{k}\otimes\id_{1})(\id_{1}\otimes
    f)(t_{1}\otimes\id_{1}^{\otimes k-1}))
  \end{align*}
  which can be diagrammatically represented as follows:
  \begin{equation*}
    \input{rotation_trace_trans.latex}
  \end{equation*}
  Then, we use the adjoint of Wenzl's formula \eqref{eq:wenzlrecursion}. The
  term with $P_{k-1}\otimes\id_{1}$ yields
  \begin{equation*}
    \Tr\left((\id_{1}^{\otimes (k-1)}\otimes t_{1}^{*})(P_{k-1}\otimes\id_{1}^{\otimes 2})(\id_{1}\otimes f)(t_{1}\otimes\id_{1}^{\otimes (k-1)})\right) = \Tr\left((P_{k-1}\otimes t_{1}^{*})(\id_{1}\otimes f)(t_{1}\otimes\id_{1}^{\otimes (k-1)})\right).
  \end{equation*}
  This vanishes because the range of $f$ is contained in $H_{k}$ and
  $\id_{1}^{\otimes(k-2)}\otimes t_{1}^{*}$ is an intertwiner to
  $H_{1}^{\otimes(k-2)}$, which contains no subrepresentation equivalent to
  $H_k$. The terms from \eqref{eq:wenzlrecursion} with $l>1$ also vanish because
  $(\id_{1}^{\otimes (l-1)}\otimes t_{1}^{*}\otimes \id_{1}^{\otimes
    (k-l-1)}\otimes t_{1}\otimes\id_{1})(\id_{1}\otimes f) = 0$ for the same
  reason as before. Hence, we are left with
  \begin{align*}
    \Tr(\rho(f)) & =  \frac{(-1)^{k-1}}{d_{k-1}}\Tr\left((P_{k-1}\otimes t_{1}^{*})(t_{1}^{*}\otimes \id_{1}^{\otimes (k-2)}\otimes t_{1}\otimes\id_{1})(\id_{1}\otimes f)(t_{1}\otimes\id_{1}^{\otimes (k-1)})\right) \\
    & = \frac{(-1)^{k-1}}{d_{k-1}}\Tr\left(P_{k-1}(t_{1}^{*}\otimes
      \id_{1}^{\otimes (k-1)})(\id_{1}\otimes f)(t_{1}\otimes\id_{1}^{\otimes
        (k-1)})\right) = \frac{(-1)^{k-1}}{d_{k-1}} \Tr(f).
  \end{align*}
  Here is the diagrammatic computation:
  \begin{equation*}
    \input{rotation_trace_last.latex}
  \end{equation*}
  For the Hilbert-Schmidt norm, we have
  \begin{align*}
    \Tr(\rho(f)^{*}\rho(f)) & = \Tr\left(t_{1}^{*}\otimes P_{k})(\id_{1}\otimes f^{*}\otimes\id_{1})(P_{k}\otimes t_{1}t_{1}^{*})(\id_{1}\otimes f\otimes\id_{1})(t_{1}\otimes P_{k}\right) \\
    & \leqslant \Tr\left((t_{1}^{*}\otimes \id_{1}^{\otimes k})(\id_{1}\otimes f^{*}\otimes\id_{1})(\id_{1}^{\otimes k}\otimes t_{1}t_{1}^{*})(\id_{1}\otimes f\otimes\id_{1})(t_{1}\otimes \id_{1}^{\otimes k})\right) \\
    & = \Tr \left((t_{1}^{*}\otimes \id_{1}^{\otimes k-1}\otimes t_{1}^{*})(\id_{1}\otimes f^{*}\otimes\id_{1}^{\otimes 2})(\id_{1}^{\otimes k}\otimes t_{1}t_{1}^{*}\otimes\id_{1})(\id_{1}\otimes f\otimes\id_{1}^{\otimes 2})(t_{1}\otimes \id_{1}^{\otimes k-1}\otimes t_{1})\right) \\
    & = \Tr\left((t_{1}^{*}\otimes \id_{1}^{\otimes k-1})(\id_{1}\otimes f^{*})(\id_{1}\otimes f) (t_{1}\otimes \id_{1}^{\otimes k-1})\right) \\
    & = \Tr(f^{*}f).
  \end{align*}
\end{proof}

\begin{prop}\label{prop:partialtrace}
  Assume that $N > 2$. Let $a, b, c\in \N$ and consider the operator
  \begin{equation*}
    x_{a, b, c} = (\id_{a}\otimes \tr_{b}\otimes \id_{c})(P_{a+b+c}) : H_{a}\otimes H_{c} \rightarrow H_{a}\otimes H_{c}.
  \end{equation*}
  Then, there exist two constants $\lambda_{a, c} > 0$ and $D_{a, c} > 0$
  depending only on $N$, $a$ and $c$ such that
  \begin{equation*}
    \|x_{a, b, c} - \lambda_{a, c} (\id_{a}\otimes\id_{c})\| \leqslant D_{a, c}q^{b}.
  \end{equation*}
  In particular $x_{a, b, c} \to \lambda_{a, c} (\id_{a}\otimes\id_{c})$ as
  $b\to\infty$.
\end{prop}

\begin{proof}
  For convenience, the proof will be done with the non-normalized trace, and
  hence we consider the non-normalized operator $X_{a, b, c} =
  (\id_{a}\otimes\Tr_{b}\otimes \id_{c})(P_{a+b+c}) = d_{b}x_{a, b, c}$. We
  first observe that
  \begin{equation*}
    (\Tr_{a}\otimes \Tr_{c})(X_{a, b, c}) = \Tr(P_{a+b+c}) = d_{a+b+c} = d_{b}q^{-a-c} + O(q^b)
  \end{equation*}
  and accordingly set
  \begin{equation*}
    \lambda_{a, c} = q^{-a-c}/d_{a}d_{c} \text{ and } X'_{a, b, c} = X_{a, b, c} - d_{b} \lambda_{a, c} (\id_{a}\otimes\id_{c}).
  \end{equation*}
  With this notation, we have $\Tr(X'_{a, b, c}) = O(q^{b})$ and we want to show
  that $\|X'_{a, b, c}\| \leqslant D_{a,c}$. We will prove that
  \begin{equation*}
    \vert (\Tr_a\otimes\Tr_c)(X'_{a,b,c}f) \vert\leqslant D_{a,c} \|f\|_{\HS}
  \end{equation*}
  for any $f \in \B(H_{a}\otimes H_{c})$. Moreover any such $f$ can be
  decomposed into a multiple of the identity and a map with zero trace, and
  since the estimate is satisfied for $f = \id$ by our choice of $\lambda_{a,
    c}$ we can assume $(\Tr_{a}\otimes\Tr_{c})(f) = 0$. Eventually, we note that
  in this case $(\Tr_{a}\otimes\Tr_{c})(X'_{a, b, c}f) =
  (\Tr_{a}\otimes\Tr_{c})(X_{a, b, c}f)$.

  Now we observe that $(\Tr_{a}\otimes \Tr_{c})(X_{a, b, c}f) =
  \Tr(P_{a+b+c}f_{13})$ where $\Tr$ is the trace of $H_1^{\otimes(a+b+c)}$, and
  we use Wenzl's formula \eqref{eq:wenzlrecursion} to write
  \begin{align*}
    \Tr(X_{a, b, c}f) = & \Tr((P_{a+b+c-1}\otimes \id_{1})f_{13}) \\
    & + \sum_{l=1}^{a+b+c-1}(-1)^{a+b+c-l}\frac{d_{l-1}}{d_{a+b+c-1}}
    \Tr\left((\id_{1}^{\otimes (l-1)}\otimes t_{1}\otimes
      \id_{1}^{\otimes(a+b+c-l-1)}\otimes t_{1}^{*})(P_{a+b+c-1}\otimes
      \id_{1})f_{13}\right).
  \end{align*}
  Moreover, one can factor $P_{a}\otimes P_{b}\otimes P_{c}$ out of the right
  side of $(P_{a+b+c-1}\otimes \id_{1})f_{13}$. Since $P_{k}(\id\otimes
  t_{1}\otimes\id) = 0$ on $H_{1}^{\otimes(k-2)}$, we see that $(P_{a}\otimes
  P_{b}\otimes P_{c})(\id_{1}^{\otimes (l-1)}\otimes t_{1}\otimes
  \id_{1}^{\otimes(a+b+c-l-1)}) = 0$ if $l\neq a$ and $l\neq a+b$. Hence there
  are only three terms to bound in the expression above.

  The first term is equal to
  \begin{equation*}
    \Tr((P_{a+b+c-1}\otimes \id_{1})f_{13}) = \Tr(X_{a, b, c-1}f^{\flat}),
  \end{equation*}
  where $f^{\flat} = (\id_{a}\otimes \id_{c-1}\otimes \Tr_{1})(f)$ satisfies
  $\Tr(f^{\flat}) = 0$ and
  $\|f^{\flat}\|_{\HS}\leqslant\sqrt{d_{1}}\|f\|_\HS$. For $l=a$, we use the
  trivial bound
  \begin{equation*}
    \frac{d_{a-1}}{d_{a+b+c-1}} \times \|f_{13}\|_{\HS} \times\|t_{1}\|^{2} \times \|P_{a+b+c-1}\otimes\id_{1}\|_{\HS} = \frac{d_{1}^{3/2}d_{a-1}\sqrt{d_{b}}}{\sqrt{d_{a+b+c-1}}} \|f\|_{\HS}.
  \end{equation*}
  For $l=a+b$, if we denote the term we are interested in by $Y$, we have, with
  $f = \sum f_{(1)}\otimes f_{(2)}\in \B(H_{a})\otimes \B(H_{c})$,
  \begin{align*}
    Y & = \Tr\left( (\id_{1}^{\otimes (a+b-1)}\otimes
      t_{1}\otimes\id_{1}^{\otimes (c-1)}\otimes t_{1}^{*}) (P_{a+b+c-1}\otimes
      \id_{1})f_{13}
    \right) \\
    & = \Tr\left( (\id_{1}^{\otimes (a+b+c-2)}\otimes
      t_{1}^{*})(P_{a+b+c-1}\otimes \id_{1}) f_{13}(\id_{1}^{\otimes
        (a+b-1)}\otimes t_{1}\otimes \id_{1}^{\otimes(c-1)})
    \right) \\
    & = \Tr\left((\id_{1}^{\otimes (a+b+c-2)}\otimes
      t_{1}^{*})(P_{a+b+c-1}\otimes (t_{1}^{*}\otimes \id_{1})(\id_{1}\otimes
      t_{1}))f_{13}(\id_{1}^{\otimes (a+b-1)}
      \otimes t_{1}\otimes \id_{1}^{\otimes(c-1)})\right) \\
    & = \Tr\left( (\id_{1}^{\otimes (a+b+c-2)}\otimes
      t_{1}^{*})([(P_{a+b+c-1}\otimes t_{1}^{*})
      (f_{13}\otimes\id_{1})(\id_{1}^{(a+b-1)}\otimes t_{1}\otimes
      \id_{1}^{\otimes c})]\otimes \id_{1})(\id_{1}^{\otimes (a+b+c-2)}\otimes
      t_{1})
    \right)\\
    & = \sum\Tr\left( (P_{a+b+c-1}\otimes t_{1}^{*})(f_{(1)}\otimes
      \id_{b-1}\otimes f_{(2)}\otimes \id_{1}) (\id_{1}^{\otimes (a+b-1)}\otimes
      t_{1}\otimes \id_{1}^{\otimes c})
    \right) \\
    & = \Tr(P_{a+(b-1)+c} f_{13}^{\sharp}) = \Tr(X_{a, b-1, c}f^{\sharp})
  \end{align*}
  where $f^{\sharp} = (\id_{a}\otimes \rho)(f)$ satisfies $\Tr(f^{\sharp}) = 0$
  and $\|f^{\sharp}\|_{\HS}\leqslant \|f\|_{\HS}$ by Lemma
  \ref{lem:tracezero}. Here is the diagrammatic version of the previous
  computation,
  \begin{equation*}
    \input{term_a+b.latex}
  \end{equation*}
  We recognize indeed $\rho(f_{(2)})$ in the last diagram. The projections
  $P_{c}$ included in the definition of $\rho(f_{(2)})$ do not appear on the
  diagram since they are absorbed by $P_{a+b+c-1}$ (through the trace for one of
  them), but they must be taken into account. Summing up, we have
  \begin{equation}\label{eq:inequality}
    \vert \Tr(X_{a, b, c}f)\vert \leqslant \vert\Tr(X_{a, b, c-1}f^{\flat})\vert + 
    \frac{d_{a+b-1}}{d_{a+b+c-1}}\vert \Tr(X_{a, b-1, c}f^{\sharp})\vert + 
    \frac{d_1^{3/2}d_{a-1}\sqrt{d_{b}}}{\sqrt{d_{a+b+c-1}}}\|f\|_\HS.
  \end{equation}

  We will now proceed by induction on $c$ with the following induction
  hypothesis
  \begin{center}
    $H(c)$: "for all $a \in \N$ there exists a constant $D_{a, c}$ such that for
    all $b\in \N$ and all $f\in \B(H_{a}\otimes H_{c})$ satisfying $\Tr(f) = 0$
    we have $\vert \Tr(X_{a, b, c}f)\vert \leqslant D_{a, c}\|f\|_{\HS}$".
  \end{center}
  Recall that $H(0)$ holds with $D_{a, c} = 0$ because $X_{a, b, 0}$ is an
  intertwiner, hence a multiple of the identity.

  Now we take $c>0$, we assume that $H(c-1)$ holds and we apply it to the first
  term in the right-hand side of Equation~\eqref{eq:inequality}. Since
  $\|f^{\flat}\|_\HS\leqslant \sqrt{d_1}\|f\|_\HS$ and $d_b \leqslant
  d_{a+b+c-1}$, this yields
  \begin{equation*}
    \vert \Tr(X_{a, b, c}f)\vert\leqslant (\sqrt{d_1}D_{a, c-1}+d_{1}^{3/2}d_{a-1})\|f\|_{\HS} + \frac{d_{a+b-1}}{d_{a+b+c-1}}\vert\Tr(X_{a, b-1, c}f^{\sharp})\vert.
  \end{equation*}
  We set $D' = \max(\sqrt{d_1}D_{a, c-1}+d_{1}^{3/2}d_{a-1},\sqrt{d_{a+c}})$ and
  we iterate the inequality above over $b$. Noticing that $\vert
  \Tr(X_{a,0,c}f^{\sharp b})\vert \leqslant \sqrt{d_{a+c}} \|f^{\sharp b}\|_\HS
  \leqslant D' \|f^{\sharp b}\|_\HS$ this yields, with the convention that the
  product equals $1$ for $l=0$:
  \begin{equation*}
    \vert \Tr(X_{a, b, c}f)\vert\leqslant D'\sum_{l=0}^{b}\|f^{\sharp l}\|
    \left(\prod_{t = b-l+1}^{b}\frac{d_{a+t-1}}{d_{a+t+c-1}}\right).
  \end{equation*}
  Using the inequality $\|f^{\sharp}\|_{\HS} \leqslant \|f\|_{\HS}$ , as well as
  the estimate $d_{x}/d_{y}\leqslant q^{y-x}$ for $x<y$ and the fact that $\vert
  q\vert < 1$ if $N>2$, we see that $H(c)$ holds:
  \begin{equation*}
    \vert \Tr(X_{a, b, c}f)\vert\leqslant D' \|f\|_{\HS} \sum_{l=0}^{b}q^{lc} \leqslant D' \|f\|_{\HS}\sum_{l=0}^{\infty}q^{lc}.
  \end{equation*} 
\end{proof}

It is clear from the beginning of the proof that the
Proposition~\ref{prop:partialtrace} has the following equivalent formulation,
which we will use for the proof of Theorem~\ref{thm:keyestimate}:

\begin{cor}\label{cor:partialtrace}
  Assume that $N>2$. For any $a$, $c \in \N$ and any $f \in \B(H_{a}\otimes
  H_{c})$ such that $\Tr(f) = 0$, there exists a constant $D_{a, c}$ such that
  we have, for any $b\in\N$:
  \begin{equation*}
    \vert \Tr(P_{a+b+c}f_{13}) \vert \leqslant D_{a,c} \|f\|_{\HS}.
  \end{equation*}
\end{cor}

\subsection{A variation on Wenzl's recursion formula}

The second result can be called a "higher weight" version of Wenzl's recursion
formula \eqref{eq:wenzlrecursionreflected}. As a matter of fact, let $\zeta =
\sum\zeta^{(1)}\otimes \zeta^{(2)}$ be a vector in $H_{2}\subset H_{1}\otimes
H_{1}$. Then, the map $f=\sum \zeta^{(2)}\overline{\zeta}^{(1)*}\in \B(H_{1})$
has trace $0$, so that applying $\Tr_{1}(f\,\cdot\,)\otimes \id_{n-1}$ to both
sides of Equation~\eqref{eq:wenzlrecursionreflected} yields
\begin{equation*}
  \sum (\overline{\zeta}_{(1)}^{*}\otimes \id_{n-1}) P_{n}(\zeta_{(2)}\otimes\id_{n-1}) = - \frac{d_{n-2}}{d_{n-1}}\sum P_{n-1}(\zeta_{(1)}\overline{\zeta}_{(2)}^{*}\otimes \id_{n-2})P_{n-1}.
\end{equation*}
What we are going to prove is a similar equality but with $\zeta$ being any
highest weight vector, i.e. $\zeta\in H_{p+q}\subset H_{p}\otimes H_{q}$ for
arbitrary $p$ and $q$.

\begin{lem}\label{lem:astuce}
  Let $\zeta\in H_{p+q}$ be decomposed as $\zeta = \sum
  \zeta^{(1)}\otimes\zeta^{(2)}\in H_{p}\otimes H_{q}$ and $\zeta = \sum
  \zeta_{(1)}\otimes\zeta_{(2)}\in H_{q}\otimes H_{p}$. For all $n\geqslant
  p+q$, there exist $\alpha_{p, q}^{n}\in \C$ such that
  \begin{align}\label{eq:first}
    \sum (\overline{\zeta}^{(1)*}\otimes
    \id_{n-p})P_{n}(\zeta^{(2)}\otimes\id_{n-q}) & = \alpha_{p, q}^{n}\sum
    P_{n-p}(\zeta_{(1)}\overline{\zeta}_{(2)}^{*}\otimes
    \id_{n-p-q})P_{n-q}\\ \label{eq:second} \sum
    (\id_{n-p}\otimes\zeta^{(1)*})P_{n}(\id_{n-q}\otimes \overline{\zeta}^{(2)})
    & = \alpha_{p, q}^{n}\sum
    P_{n-p}(\id_{n-p-q}\otimes\overline{\zeta}_{(1)}\zeta_{(2)}^{*})P_{n-q}.
  \end{align}
  Moreover, there exist constants $C_{p, q}>0$ such that for all $n\in \N$,
  $C_{p, q}\leqslant \vert\alpha_{p, q}^{n}\vert\leqslant 1$.
\end{lem}

\begin{proof}
  Let us first note that the second equality follows from the first one by
  conjugation, hence we will only focus on the first one. If $p=0$, then
  \begin{equation*}
    P_{n}(\zeta^{(2)}\otimes \id_{n-q}) =  P_{n}(\zeta^{(2)}\otimes \id_{n-q})P_{n-q} = P_{n}(\zeta_{(1)}\otimes \id_{n-q})P_{n-q}
  \end{equation*}
  and the result is proved with $\alpha_{0, q}^{n} = 1$ for all $n$. Similarly,
  the result holds for $q=0$ with $\alpha_{p, 0}^{n} = 1$. We will proceed by
  induction on $p$ and $q$ with the induction hypothesis
  \begin{center}
    $H_{N}$ : "For any $p, q$ with $p+q\leqslant N$, there exists a constant
    $C_{p, q}>0$ such that for all $n\geqslant p+q$, \\ there is a constant
    $\alpha_{p, q}^{n}$ such that Equations \eqref{eq:first} and
    \eqref{eq:second} hold and $C_{p, q}\leqslant \vert\alpha_{p,
      q}^{n}\vert\leqslant 1$."
  \end{center}
  As we have seen, $H_{0}$ and $H_{1}$ hold, so let us assume $H_{N}$ and
  consider $p, q\geqslant 1$ such that $p+q=N+1$. In order to use the induction
  hypothesis, we refine the decompositions of $\zeta$ in the following way:
  \begin{align*}
    \zeta^{(1)} & = \sum \zeta^{(11)}\otimes \zeta^{(12)}\in H_{p-1}\otimes H_{1}, \\
    \zeta^{(1)} & = \sum \zeta^{(1)}_{(1)}\otimes \zeta^{(1)}_{(2)}\in H_{1}\otimes H_{p-1}, \\
    \zeta^{(2)} & = \sum \zeta^{(21)}\otimes \zeta^{(22)}\in H_{1}\otimes H_{q-1}, \\
    \zeta^{(2)} & = \sum \zeta^{(2)}_{(1)}\otimes \zeta^{(2)}_{(2)}\in
    H_{q-1}\otimes H_{1}.
  \end{align*}
  Applying the map $\sum (\overline{\zeta}^{(1)*}\otimes \id_{n-p}) (\,\cdot\,)
  (\zeta^{(2)}\otimes \id_{n-q})$ to Wenzl's formula
  \eqref{eq:wenzlrecursionreflected}, the first term on the right-hand side
  reads
  \begin{align*}
    & \sum (\overline{\zeta}^{(12)*}\otimes \overline{\zeta}^{(11)*}\otimes
    \id_{n-p})(\id_{1}\otimes P_{n-1})(\zeta^{(21)}\otimes \zeta^{(22)}\otimes
    \id_{n-q})
    \\
    = & \sum
    \overline{\zeta}^{(12)*}(\zeta^{(21)})(\overline{\zeta}^{(11)*}\otimes
    \id_{n-p})P_{n-1}(\zeta^{(22)}\otimes \id_{n-q}).
  \end{align*}
  Consider the linear map $T : H_{p-1}\otimes H_{q-1} \rightarrow \B(H_{n-q},
  H_{n-p})$ defined by $T(x\otimes y) =
  (\overline{x}^*\otimes\id_{n-p})P_{n-1}(y\otimes \id_{n-q})$. Then, the term
  above equals
  \begin{equation*}
    T\left(\sum \overline{\zeta}^{(12)*}(\zeta^{(21)})(\zeta^{(11)}\otimes \zeta^{(22)})\right) = 
    T\left((\id_{p-1}\otimes t_{1}^{*}\otimes\id_{q-1})(\zeta)\right).
  \end{equation*}
  The argument of $T$ on the right-hand side vanishes because $\zeta$ is a
  highest weight vector, so that the whole term vanishes. Coming back
  to~\eqref{eq:wenzlrecursionreflected} and setting $L =
  \sum(\overline{\zeta}^{(1)*}\otimes \id_{n-p})P_{n}(\zeta^{(2)}\otimes
  \id_{n-q})$, we thus have
  \begin{align*}
    L & = -\frac{d_{n-2}}{d_{n-1}}\sum (\overline{\zeta}^{(1)*}\otimes \id_{n-p})(\id_{1}\otimes P_{n-1})(t_{1}t_{1}^{*}\otimes \id_{n-2})(\id_{1}\otimes P_{n-1})(\zeta^{(2)}\otimes \id_{n-q}) \\
    & = -\frac{d_{n-2}}{d_{n-1}}\sum (\overline{\zeta}^{(11)*}\otimes \id_{n-p})P_{n-1}(\overline{\zeta}^{(12)*}\otimes \id_{n-1})(t_{1}t_{1}^{*}\otimes \id_{n-2})(\zeta^{(21)}\otimes \id_{n-1})P_{n-1}(\zeta^{(22)}\otimes \id_{n-q}) \\
    & = -\frac{d_{n-2}}{d_{n-1}}\sum
    (\overline{\zeta}^{(11)*}\otimes\id_{n-p})P_{n-1}(\zeta^{(12)}\overline{\zeta}^{(21)*}\otimes\id_{n-2})P_{n-1}(\zeta^{(22)}\otimes
    \id_{n-q}).
  \end{align*}
  Now we apply $H_{N}$ to $\zeta^{(1)}$ (with $p'=p-1$, $q'=1$) and to
  $\zeta^{(2)}$ (with $p'=1$, $q'=q-1$) to get
  \begin{equation*}
    L = -\frac{d_{n-2}}{d_{n-1}}\alpha_{p-1, 1}^{n-1}\alpha_{1, q-1}^{n-1}\sum \left(P_{n-p}(\zeta^{(1)}_{(1)}\overline{\zeta}^{(1)*}_{(2)}\otimes \id_{n-p-1})P_{n-2}\right)
    \left(P_{n-2}(\zeta^{(2)}_{(1)}\overline{\zeta}^{(2)*}_{(2)}\otimes \id_{n-q-1})P_{n-q}
    \right).
  \end{equation*}
  The last step is to apply again the induction hypothesis. To do this, we need
  to refine once more our decomposition by setting
  \begin{align*}
    \zeta & = \sum \eta^{(1)}\otimes \eta^{(2)}\otimes \eta^{(3)} \in H_{1}\otimes H_{p+q-2}\otimes H_{1} \\
    \eta^{(2)} & = \sum \eta^{(21)}\otimes \eta^{(22)} \in H_{p-1}\otimes H_{q-1} \\
    \eta^{(2)} & = \sum \eta^{(2)}_{(1)}\otimes \eta^{(2)}_{(2)} \in
    H_{q-1}\otimes H_{p-1}.
  \end{align*}
  Note that in the above computations we can replace everywhere
  $\zeta^{(1)}_{(1)}$, $\zeta^{(1)}_{(2)}$, $\zeta^{(2)}_{(1)}$ and
  $\zeta^{(2)}_{(2)}$ respectively by $\eta^{(1)}$, $\eta^{(21)}$, $\eta^{(22)}$
  and $\eta^{(3)}$. Thus, applying $H_{N}$ to $\eta^{(2)}$ (with $p'=p-1$,
  $q'=q-1$) yields
  \begin{align*}
    L & = -\frac{d_{n-2}}{d_{n-1}}\alpha_{p-1, 1}^{n-1}\alpha_{1, q-1}^{n-1}\alpha_{p-1, q-1}^{n-2}\sum P_{n-p}(\eta^{(1)}\otimes \id_{n-p-1}) \\
    & \makebox[2cm]{} P_{n-p-1}(\eta^{(2)}_{(1)}\overline{\eta}^{(2)*}_{(2)}\otimes \id_{n-p-q})P_{n-q-1}(\overline{\eta}^{(3)*}\otimes \id_{n-q-1})P_{n-q} \\
    & = -\frac{d_{n-2}}{d_{n-1}}\alpha_{p-1, 1}^{n-1}\alpha_{1, q-1}^{n-1}\alpha_{p-1, q-1}^{n-2}\sum P_{n-p}\left((\eta^{(1)}\otimes \eta^{(2)}_{(1)})(\overline{\eta}^{(3)*}\otimes \overline{\eta}^{(2)*}_{(2)})\otimes \id_{n-p-q} \right)P_{n-q} \\
    & = -\frac{d_{n-2}}{d_{n-1}}\alpha_{p-1, 1}^{n-1}\alpha_{1,
      q-1}^{n-1}\alpha_{p-1, q-1}^{n-2}\sum
    P_{n-p}(\zeta_{(1)}\overline{\zeta}_{(2)}^{*}\otimes \id_{n-p-q})P_{n-q}.
  \end{align*}
  This proves Equation \eqref{eq:first} for $p$ and $q$ and as mentioned at the
  beginning of the proof, Equation \eqref{eq:second} follows by
  conjugation. Moreover, we see that
  \begin{equation*}
    \vert\alpha_{p, q}^{n}\vert \geqslant 
    \frac{d_{n-2}}{d_{n-1}}C_{p-1, 1}C_{1, q-1}C_{p-1, q-1} \geqslant \frac{1}{d_{1}}C_{p-1, 1}C_{1, q-1}C_{p-1, q-1} > 0
  \end{equation*}
  hence $H_{N+1}$ holds and the proof is complete.
\end{proof}

\section{The key estimate}\label{sec:estimate}

We now turn to the main technical result of this article,
Theorem~\ref{thm:keyestimate}, which concerns the behavior of the scalar
product $\langle \chi_{l}u_{\xi'\eta'}^{k}, u_{\xi\eta}^{n}\chi_{l'}\rangle$ as
$l$, $l'$ tend to $+\infty$. Its proof will span the whole of this section.

We start by recalling two technical lemmata from the literature on free
orthogonal quantum groups. The first one gives a norm estimate for some explicit
intertwiners in tensor products of irreducible representations. For any four
integers $l$, $k$, $m$ and $a$ such that $k+l = m+2a$, the map
\begin{equation*}
  \left(V_{m}^{l, k}\right)^{*} =  P_{m}(\id_{l-a}\otimes t_{a}^{*}\otimes \id_{k-a})
\end{equation*}
is an intertwiner from $H_{l}\otimes H_{k}$ to $H_{m}$, hence there is a scalar
$\kappa_{m}^{l, k}$ such that $v_{m}^{l, k} = \kappa_{m}^{l, k} V_{m}^{l, k}$ is
an isometric intertwiner. The scalar $\kappa_{m}^{l, k}$ can be explicitly
computed, see \cite{vergnioux2007property}. However, we will only need the
following consequence of this computation.

\begin{lem}\label{lem:kappa}
  There exists a constant $B_{a}$, depending only on $a$ and $N$, such that for
  all $k$, $l$ and $m = k+l-2a$ we have $\left\vert\kappa_{m}^{l,
      k}\right\vert\leqslant B_{a}$.
\end{lem}

\begin{proof}
  This is a consequence of the estimates given in \cite[Lem
  4.8]{vergnioux2007property}, see also \cite{fima2014cocycle}. The sequence
  $(B_{a})_{a}$ diverges exponentially as $q^{-a/2}$.
\end{proof}

We will also need the following estimates which were already used in
\cite{vaes2007boundary} and \cite{fima2014cocycle}.

\begin{lem}\label{lem:productprojections}
  Let $x$, $y$ and $z$ be integers and let $\mu\neq x+y+z$ be a
  subrepresentation of both $x\otimes (y+z)$ and $(x+y)\otimes z$. Then, there
  exists a constant $A>0$ depending only on $N$ such that
  \begin{equation*}
    \|(\id_{x}\otimes P_{y+z})(P_{x+y}\otimes \id_{z})-P_{x+y+z}\| \leqslant Aq^{y} \text{ and } \|P^{x, y+z}_{\mu}P_{\mu}^{x+y, z}\|\leqslant Aq^{y}.
  \end{equation*}
\end{lem}

\begin{proof}
  The first inequation is \cite[Lem A.4]{vaes2007boundary}.  For the second one,
  note that $P_{\mu}^{x, y+z}P_{x+y+z} = 0 = P_{x+y+z}P_{\mu}^{x+y, z}$ because
  $\mu$ is not the highest weight. Thus, we have
  \begin{eqnarray*}
    \|P^{x, y+z}_{\mu}P_{\mu}^{x+y, z}\| & = & \|P^{x, y+z}_{\mu}\left((\id_{x}\otimes P_{y+z})(P_{x+y}\otimes \id_{z})-P_{x+y+z}\right)P_{\mu}^{x+y, z}\| \\
    & \leqslant & \|P^{x, y+z}_{\mu}\|\|(\id_{x}\otimes P_{y+z})(P_{x+y}\otimes \id_{z})-P_{x+y+z}\|\|P_{\mu}^{x+y, z}\| \\
    & \leqslant & \|(\id_{x}\otimes P_{y+z})(P_{x+y}\otimes \id_{z})-P_{x+y+z}\| \\
    & \leqslant & Aq^{y}.
  \end{eqnarray*}
\end{proof}

We now state and prove an estimate, as $l$, $l'$ tend to $+\infty$, about the
scalar product between products of the characters $\chi_l$, $\chi_{l'}$ with
coefficients of fixed representations. Since $\chi_l$, $\chi_{l'}$ have norm $1$
in the GNS space $L^2(\G)$, it is clear that these scalar products are bounded
when $l$, $l'$ tend to $+\infty$. However one can do much better:

\begin{thm}\label{thm:keyestimate}
  Assume that $N>2$. Let $k$, $n$ be integers, let $\xi, \eta\in H_{n}$ be
  \emph{orthogonal} unit vectors and let $\xi', \eta'\in H_{k}$ be arbitrary
  unit vectors. Then, there exists $K>0$ such that we have, for all integers
  $l$, $l'$:
  \begin{equation*}
    \left\vert\left\langle \chi_{l}u_{\xi'\eta'}^{k}, u_{\xi\eta}^{n}\chi_{l'}\right\rangle\right\vert \leqslant K q^{\max(l, l')}.
  \end{equation*}
  In particular $\left\vert\left\langle \chi_{l}u_{\xi'\eta'}^{k},
      u_{\xi\eta}^{n}\chi_{l'}\right\rangle\right\vert \to 0$ when $l$ or $l'$
  tends to $+\infty$.
\end{thm}

\begin{proof}
  The proof will consist of the following steps:
  \begin{enumerate}
  \item[\textbf{1.}] computation of the scalar product as a sum $S = \sum S_m$
    in the category of representations,
  \item[\textbf{2.}] simplification of $S_m$ into $T_m$,
  \item[\textbf{3.}] expression of $T_m$ as a trace,
  \item[\textbf{4.}] application of Lemma~\ref{lem:astuce} to reduce the trace,
  \item[\textbf{5.}] application of Proposition~\ref{prop:partialtrace} to
    estimate the trace,
  \item[\textbf{6.}] backtracking of all approximations.
  \end{enumerate}

  \bigskip
  \noindent \textbf{Step 1.} We compute the products and the scalar product
  using the formul\ae{} given in Subsection \ref{subsec:irreducible}:
  \begin{align}
    S &= \left\langle \chi_{l}u_{\xi'\eta'}^{k}, u_{\xi\eta}^{n}\chi_{l'}\right\rangle = \sum_{i=1}^{d_{l}}\sum_{j=1}^{d_{l'}} \left\langle u_{e_{i}e_{i}}^{l}u_{\xi'\eta'}^{k}, u_{\xi\eta}^{n} u_{e_{j}e_{j}}^{l'} \right\rangle \nonumber \\
    & = \sum_{i=1}^{d_{l}}\sum_{j=1}^{d_{l'}}\sum_{m=0}^{+\infty}
    \left\langle u^{m}_{v_{m}^{l, k *}(e_{i}\otimes \xi'), v_{m}^{l, k *}(e_{i}\otimes\eta')}, u^{m}_{v_{m}^{n, l'*}(\xi\otimes e_{j}), v_{m}^{n, l' *}(\eta\otimes e_{j})} \right\rangle\nonumber \\
    & = \sum_{i=1}^{d_{l}}\sum_{j=1}^{d_{l'}}\sum_{m=0}^{+\infty} \frac 1{d_{m}}
    \left\langle v_{m}^{l, k*}(e_{i}\otimes\xi'), v_{m}^{n, l' *}(\xi\otimes
      e_{j})\right\rangle
    \left\langle v_{m}^{n, l'*}(\eta\otimes e_{j}), v_{m}^{l, k*}(e_{i}\otimes\eta')\right\rangle \nonumber \\
    & = \sum_{i=1}^{d_{l}}\sum_{j=1}^{d_{l'}}\sum_{m=0}^{+\infty} \frac 1{d_{m}}
    \left\langle v_{m}^{l, k*}(e_{i}\otimes\xi'), v_{m}^{n, l'*}(\xi\otimes
      e_{j})\right\rangle
    \left\langle v_{m}^{k, l*}(\overline \eta' \otimes\overline e_{i}), v_{m}^{l', n *}(\overline e_{j} \otimes \overline\eta) \right\rangle \nonumber \\
    & = \sum_{m=0}^{+\infty}\frac{1}{d_{m}} \left\langle \left(v_{m}^{l,
          k}\otimes v_{m}^{k, l}\right)^{*}\circ\left(\Sigma\otimes
        \Sigma\right)(\xi'\otimes t_{l}\otimes \overline{\eta}'),
      \left(v_{m}^{n, l'}\otimes v_{m}^{l', n}\right)^{*}(\xi\otimes
      t_{l'}\otimes \overline{\eta})\right\rangle.
    \label{eq:Sdefinition}
  \end{align}
  Let us denote by $S^{m}$ the $m$-th term in brackets in \eqref{eq:Sdefinition}
  and note that it can only be non-zero if $u^{m}$ is a subrepresentation of
  both $u^{k}\otimes u^{l}$ and $u^{n}\otimes u^{l'}$. This means that there are
  integers $a$ and $b$ such that
  \begin{equation*}
    l+k = m+2a \text{ and } n+l' = m+2b.
  \end{equation*}
  Note that $l-n+b-a = l'-k+a-b$ and let us denote by $c$ this number. To
  estimate $S^{m}$, we will use the explicit formula for the intertwiners given
  just before Lemma \ref{lem:kappa}:
  \begin{align*}
    \left(v_{m}^{l, k}\right)^{*} = \kappa_m^{kl} P_{m}(\id_{l-a}\otimes
    t_{a}^{*}\otimes \id_{k-a})
    &,~~ \left(v_{m}^{k, l}\right)^{*} = \kappa_m^{kl} P_{m}(\id_{k-a}\otimes t_{a}^{*}\otimes \id_{l-a}), \\
    \left(v_{m}^{l', n}\right)^{*} = \kappa_m^{nl'} P_{m}(\id_{l'-b}\otimes
    t_{b}^{*}\otimes \id_{n-b}) &,~~ \left(v_{m}^{n, l'}\right)^{*} =
    \kappa_m^{nl'}P_{m}(\id_{n-b}\otimes t_{b}^{*}\otimes \id_{l'-b}).
  \end{align*}
  so that \eqref{eq:Sdefinition} becomes:
  \begin{align*}
    S^{m} = \left(\kappa_m^{kl}\right)^{2} \left(\kappa_m^{nl'}\right)^{2} &
    \left\langle (P_{m}\otimes P_{m})\left(\id_{l-a}\otimes t_{a}^{*}\otimes \id_{k-a}^{\otimes 2}\otimes t_{a}^{*}\otimes \id_{l-a}\right)(\Sigma\otimes\Sigma)(\xi'\otimes t_l \otimes \overline{\eta}')\right., \\
    &~ \left.(P_{m}\otimes P_{m})\left(\id_{n-b}\otimes
        t_{b}^{*}\otimes\id_{l'-b}^{\otimes 2}\otimes t_{b}^{*}\otimes
        \id_{n-b}\right)(\xi\otimes t_{l'}\otimes \overline{\eta})\right\rangle.
  \end{align*}

  \bigskip
  \noindent\textbf{Step 2.} Let us set, for $0\leqslant \mu, \mu'\leqslant m$,
  \begin{align*}
    S^{m}_{\mu, \mu'} = & \left\langle(P_{\mu}^{l-a, k-a}\otimes P_{\mu'}^{k-a,
        l-a})\left(\id_{l-a}\otimes t_{a}^{*}\otimes \id_{k-a}^{\otimes 2}
        \otimes t_{a}^{*}\otimes \id_{l-a}\right)(\Sigma\otimes\Sigma)(\xi'\otimes t_l \otimes \overline{\eta}')\right., \\
    &~~ \left.(P_{\mu}^{n-b, l'-b}\otimes
      P_{\mu'}^{l'-b,n-b})\left(\id_{n-b}\otimes
        t_{b}^{*}\otimes\id_{l'-b}^{\otimes 2}\otimes t_{b}^{*}\otimes
        \id_{n-b}\right)(\xi\otimes t_{l'}\otimes \overline{\eta})\right\rangle
  \end{align*}
  so that $S^{m} = (\kappa_m^{kl}\kappa_m^{nl'})^{2}S^{m}_{m, m}$. If $\mu$ or
  $\mu'$ is strictly less than $m$, then we know by Lemma
  \ref{lem:productprojections} that there is a constant $A$ depending only on
  $N$ such that either
  \begin{equation*}
    \left\|P_{\mu}^{l-a, k-a}P_{\mu}^{n-b, l'-b}\right\| \leqslant Aq^{l-a - (n-b)} \text{ or } 
    \left\|P_{\mu'}^{k-a, l-a}P_{\mu'}^{l'-b, n-b}\right\| \leqslant Aq^{l-a - (n-b)}.
  \end{equation*}
  This gives the bound $\vert S_{\mu, \mu'}^{m}\vert \leqslant A
  \|t_l\|\|t_{l'}\| \|t_a\|^2 \|t_b\|^2 q^{c} =
  A\sqrt{d_{l}}\sqrt{d_{l'}}d_ad_bq^{c}$ which will be used in the end to
  estimate $S^{m}$. Let us expand back the vectors $t_{l} = \sum
  e_{t}^{l}\otimes\overline{e}_{t}^{l}$ and $t_{l'} = \sum
  e_{s}^{l'}\otimes\overline{e}_{s}^{l'}$ and introduce
  \begin{align*}
    T^{m} = \sum_{t=1}^{d_{l}}\sum_{s=1}^{d_{l'}} &
    \left\langle\left(\id_{l-a}\otimes t_{a}^{*}\otimes \id_{k-a}^{\otimes
          2}\otimes t_{a}^{*}\otimes
        \id_{l-a}\right)(e_{t}^{l\phantom{'}}\otimes \xi'\otimes
      \overline{\eta}'
      \otimes\overline{e}_{t}^{l})\right., \\
    &~~ \left.\left(\id_{n-b}\otimes t_{b}^{*}\otimes\id_{l'-b}^{\otimes
          2}\otimes t_{a}^{*}\otimes \id_{n-b})\right)(\xi\otimes
      e_{s}^{l'}\otimes \overline{e}_{s}^{l'}
      \otimes\overline{\eta})\right\rangle
  \end{align*}
  so that $S^{m} = (\kappa_{m}^{kl})^{2}(\kappa_{m}^{nl'})^{2} (T^{m} - \sum
  S^{m}_{\mu, \mu'})$, where the sum runs over all $(\mu, \mu')\neq(m, m)$.

  \bigskip \textbf{\noindent Step 3.} The problem is now to estimate $T^{m}$,
  using the following tensor decomposition of the vectors $\xi$, $\eta$, $\xi'$
  and $\eta'$ in Sweedler's notation:
  \begin{align*}
    \xi = \sum \xi_{(1)}\otimes \xi_{(2)} & \in H_{n-b}\otimes H_{b}, \\
    \eta = \sum \eta_{(1)}\otimes \eta_{(2)} & \in H_{n-b}\otimes H_{b}, \\
    \xi' = \sum \xi'_{(1)}\otimes \xi'_{(2)} & \in H_{a}\otimes H_{k-a}, \\
    \eta' = \sum \eta'_{(1)}\otimes \eta'_{(2)} & \in H_{a}\otimes H_{k-a}.
  \end{align*}
  Because $t_{a}^{*}(x\otimes \overline{y}) = y^{*}(x)$, we get
  \begin{align*}
    T^{m} & = \sum_{t=1}^{d_{l}}\sum_{s=1}^{d_{l'}}\sum
    \left\langle\left(\id_{l-a}\otimes \overline{\xi}_{(1)}^{\prime *}\right)(e_{t}^{l})\otimes \xi_{(2)}'\otimes \overline{\eta}_{(2)}'\otimes \left(\eta_{(1)}^{\prime *}\otimes \id_{l-a}\right)(\overline{e}^{l\phantom{'}}_{t})\right., \\
    & \hspace{2.5cm}\left.\xi_{(1)}\otimes\left(\overline{\xi}_{(2)}^{*}\otimes \id_{l'-b}\right)(e_{s}^{l'})\otimes \left(\id_{l'-b}\otimes \eta_{(2)}^{*}\right)(\overline{e}^{l'}_{s})\otimes \overline{\eta}_{(1)}\right\rangle \\
    & = \sum_{t=1}^{d_{l}}\sum_{s=1}^{d_{l'}}\sum
    \left\langle\left(\xi_{(1)}^{*}\otimes \id_{l-a-(n-b)}\otimes \overline{\xi}_{(1)}^{\prime *}\right)(e_{t}^{l})\otimes \left(\eta_{(1)}^{\prime *}\otimes \id_{l-a-(n-b)}\otimes \overline{\eta}_{(1)}^{*}\right)(\overline{e}_{t}^{l\phantom{'}})\right., \\
    & \hspace{2.5cm} \left.\left(\overline{\xi}_{(2)}^{*}\otimes \id_{l'-b-(k-a)}\otimes \xi_{(2)}^{\prime *}\right)(e_{s}^{l'})\otimes \left(\overline{\eta}_{(2)}^{\prime *}\otimes \id_{l'-b-(k-a)}\otimes \eta_{(2)}^{*}\right)(\overline{e}_{s}^{l'})\right\rangle \\
    & = \sum_{t=1}^{d_{l}}\sum_{s=1}^{d_{l'}}\sum
    \left\langle\left(\xi_{(1)}^{*}\otimes \id_{c}\otimes \overline{\xi}_{(1)}^{\prime *}\right)(e_{t}^{l}), \left(\overline{\xi}_{(2)}^{*}\otimes \id_{c}\otimes \xi_{(2)}^{\prime *}\right)(e_{s}^{l'})\right\rangle \\
    & \hspace{1.9cm} \times \left\langle\left(\eta_{(1)}^{\prime *}\otimes
        \id_{c}\otimes \overline{\eta}_{(1)}^{*}\right)(\overline{e}_{t}^{l}),
      \left(\overline{\eta}_{(2)}^{\prime *}\otimes \id_{c}\otimes
        \eta_{(2)}^{*}\right)(\overline{e}_{s}^{l'}) \right\rangle.
  \end{align*}
  The properties of conjugate vectors imply that
  \begin{equation*}
    \left\langle\left(\eta_{(1)}^{\prime *}\otimes \id_{c}\otimes \overline{\eta}_{(1)}^{*}\right)(\overline{e}_{t}^{l}), \left(\overline{\eta}_{(2)}^{\prime *}\otimes \id_{c}\otimes \eta_{(2)}^{*}\right)(\overline{e}_{s}^{l'})\right\rangle = 
    \left\langle\left(\overline{\eta}_{(2)}^{*}\otimes \id_{c}\otimes \eta_{(2)}^{\prime *}\right)(e_{s}^{l'}), \left(\eta_{(1)}^{*}\otimes \id_{c}\otimes \overline{\eta}_{(1)}^{\prime *}\right)(e_{t}^{l})\right\rangle.
  \end{equation*}
  Making this change in the last expression of $T^{m}$ and using the fact that
  $\sum\langle x, Se_{s}^{l'}\rangle\langle Te_{s}^{l'}, y\rangle = \langle x,
  SP_{l'}T^{*}y\rangle$ enables to simplify the sum over $s$, we obtain
  \begin{align*}
    T^{m} & = \sum\sum_{t=1}^{d_{l}} \left\langle\left(\xi_{(1)}^{*}\otimes
        \id_{c}\otimes \overline{\xi}_{(1)}^{\prime *}\right)(e_{t}^{l}),
    \right.  \left.\left(\overline{\xi}_{(2)}^{*}\otimes
        \id_{c}\otimes\xi_{(2)}^{\prime *}\right) P_{l'}
      \left(\overline{\eta}_{(2)}\otimes\id_{c}\otimes
        \eta_{(2)}^{\prime}\right) \left(\eta_{(1)}^{*}\otimes\id_{c}\otimes
        \overline{\eta}_{(1)}^{\prime *}\right)(e_{t}^{l})
    \right\rangle \\
    & = \sum\Tr_{\otimes
      l}\left[P_{l}\left(\xi_{(1)}\overline{\xi}_{(2)}^{*}\otimes
        \id_{c}\otimes\overline{\xi}'_{(1)}\xi_{(2)}^{\prime
          *}\right)P_{l'}\left(\overline{\eta}_{(2)}\eta_{(1)}^{*}\otimes\id_{c}\otimes
        \eta_{(2)}'\overline{\eta}_{(1)}^{\prime *}\right)\right],
  \end{align*}
  where $\Tr_{\otimes l}$ denotes the non-normalized trace on $H_{1}^{\otimes
    l}$.

  \bigskip
  \noindent\textbf{Step 4.} We cannot apply Corollary \ref{cor:partialtrace} to
  $T^{m}$ because there are two highest weight projections instead of one. We
  will therefore use Lemma \ref{lem:astuce} to reduce the problem to a case
  where Corollary \ref{cor:partialtrace} applies. Let us first simplify the
  notation by setting
  \begin{align*}
    f & = \sum \xi_{(1)}\overline{\xi}_{(2)}^{*} : H_{b} \rightarrow H_{n-b}, \\
    g & = \sum \overline{\eta}_{(2)}\eta_{(1)}^{*} : H_{n-b} \rightarrow H_{b}, \\
    f' & = \sum \overline{\xi}_{(1)}^{\prime}\xi_{(2)}^{\prime *} : H_{k-a} \rightarrow H_{a}, \\
    g' & = \sum \eta'_{(2)}\overline{\eta}_{(1)}^{\prime *} : H_{a} \rightarrow
    H_{k-a}.
  \end{align*}
  By Lemma \ref{lem:productprojections}, $\|(\id_{b}\otimes
  P_{l'-b})(P_{l'-k+a}\otimes \id_{k-a}) - P_{l'}\| \leqslant Aq^{c}$ and
  $\|(P_{l-a}\otimes \id_{a})(\id_{n-b}\otimes P_{l-n+b}) - P_{l}\| \leqslant
  Aq^{c}$, so that it is enough to study
  \begin{align*}
    Y^{m} & = \Tr_{\otimes l}\left[
      (P_{l-a}\otimes \id_{a})(\id_{n-b}\otimes P_{l-n+b})(f\otimes \id_{c}\otimes f')\right. \\
    & \hspace{1.4cm} \left.(\id_{b}\otimes P_{l'-b})(P_{l'-k+a}\otimes \id_{k-a})(g\otimes \id_{c}\otimes g')\right] \\
    & = \Tr_{\otimes l}\left[
      (P_{l-a}\otimes \id_{a})(f\otimes \id_{l-n+b})(\id_{b}\otimes P_{l-n+b})(\id_{l-n+2b-a}\otimes f')\right. \\
    & \hspace{1.4cm}  \left.(\id_{b}\otimes P_{l'-b})(\id_{l'-k+a}\otimes g')(P_{l'-k+a}\otimes \id_{a})(g\otimes \id_{l-n+b})\right] \\
    & = \Tr_{\otimes l-n+2b}\left[
      (\id_{b}\otimes P_{l-n+b})(\id_{l-n+2b-a}\otimes f')(\id_{b}\otimes P_{l'-b})(\id_{l'-k+a}\otimes g')\right. \\
    & \hspace{2.3cm} \left.(P_{l'-k+a}\otimes \id_{a})(g\otimes
      \id_{l-n+b})(P_{l-a}\otimes \id_{a})(f\otimes \id_{l-n+b})\right].
  \end{align*}
  We now apply Lemma \ref{lem:astuce} to $f'$ (with $p=k-a$ and $q=a$) and $g$
  (with $p=n-b$ and $q=b$):
  \begin{align*}
    P_{l-n+b}(\id_{c}\otimes f')P_{l'-b} & = \sum (\alpha_{k-a, a}^{l'+a-b})^{-1}(\id_{l-n+b}\otimes \xi^{\prime (1) *})P_{l'+a-b}(\id_{l'-b}\otimes \overline{\xi}^{\prime (2)}) \\
    P_{l'-k+a}(g\otimes\id_{c})P_{l-a} & = \sum (\alpha_{n-b,
      b}^{l-a+b})^{-1}(\eta^{(2)
      *}\otimes\id_{l'-k+a})P_{l+b-a}(\overline{\eta}^{(1)}\otimes \id_{l-a})
  \end{align*}
  where $\xi' = \sum \xi^{\prime(1)}\otimes \xi^{\prime(2)}\in H_{k-a}\otimes
  H_{a}$ and $\eta = \sum \eta^{(1)}\otimes \eta^{(2)}\in H_{b}\otimes
  H_{n-b}$. This yields
  \begin{align*}
    Y^{m} & = \beta\sum\Tr_{\otimes l-n+2b}\left[
      (\id_{l-n+2b}\otimes \xi^{\prime (1) *})(\id_{b}\otimes P_{l'+a-b})(\id_{l'}\otimes \overline{\xi}^{\prime (2)})(\id_{l'-k+a}\otimes g')\right. \\
    &  \hspace{3.3cm} \left.(\eta^{(2) *}\otimes \id_{l'-k+2a})(P_{l+b-a}\otimes \id_{a})(\overline{\eta}^{(1)}\otimes \id_{l})(f\otimes \id_{l-n+b})\right] \\
    & = \beta\sum\Tr_{\otimes l-n+2b}\left[
      (\eta^{(2) *}\otimes\id_{l'-k+2a}\otimes \xi^{\prime (1) *})(\id_{n}\otimes P_{l'+a-b})\right. \\
    & \hspace{3.3cm}   \left.(P_{l+b-a}\otimes \id_{k})(\overline{\eta}^{(1)}\otimes f\otimes \id_{l-n+b}\otimes g'\otimes \overline{\xi}^{\prime (2)})\right] \\
    & = \beta\Tr_{\otimes c+k+n}\left[(\id_{n}\otimes
      P_{l'+a-b})(P_{l+b-a}\otimes \id_{k})(\tilde{g}\otimes f\otimes
      \id_{c}\otimes g'\otimes \tilde{f}')\right]
  \end{align*}
  where $\beta = (\alpha_{k-a, a}^{l'+a-b}\alpha_{n-b, b}^{l-a+b})^{-1}$ and
  \begin{align*}
    \tilde{f}' & = \sum \overline{\xi}^{\prime (2)}\xi^{\prime (1) *} : H_{k-a} \to H_{a}, \\
    \tilde{g} & = \sum \overline{\eta}^{(1)}\eta^{(2) *} : H_{n-b} \to H_{b}.
  \end{align*}
  To conclude the computation, we use again Lemma~\ref{lem:productprojections}
  to get the following bound:
  \begin{equation*}
    \left\|(\id_{n}\otimes P_{l'+a-b})(P_{l+b-a}\otimes \id_{k}) - P_{l+k+b-a}\right\| \leqslant Aq^{c},
  \end{equation*}
  enabling us to eventually reduce the problem to the study of
  \begin{equation*}
    Z^{m} = \beta\Tr_{\otimes c+k+n}\left[ 
      P_{l+k+b-a}(\tilde{g}\otimes f\otimes \id_{c}\otimes g'\otimes \tilde{f}')
    \right].
  \end{equation*}

  \bigskip
  \noindent \textbf{Step 5.} We will now apply
  Corollary~\ref{cor:partialtrace}. The orthogonality assumption in the
  statement of the present theorem can by rephrased as the vanishing of trace
  required for Corollary~\ref{cor:partialtrace}. We have indeed
  \begin{align*}
    \Tr(P_{n}(\tilde{g}\otimes f)) &= (\Tr_{n-b}\otimes\Tr_b)
    \left[P_{n}((\id_{b}\otimes \eta^{*})(t_{b}\otimes
      \id_{n-b})\otimes(\id_{n-b}\otimes
      t_{b}^{*})(\xi\otimes \id_{b}))\right] \\
    &= \Tr_{n-b}  \left[(\id_{n-b}\otimes t_b^*)(P_{n}\otimes\id_b) \right. \\
    & \makebox[3cm]{} \left. ((\id_{b}\otimes \eta^{*})(t_{b}\otimes
      \id_{n-b})\otimes(\id_{n-b}\otimes t_{b}^{*})(\xi\otimes \id_{b})
      \otimes\id_b)
      (\id_{n-b}\otimes t_b)\right] \\
    &= \Tr_{n-b}\left[(\id_{n-b}\otimes
      t_b^*)(P_n\otimes\id_b)((\id_b\otimes\eta^*)(t_b\otimes\id_{n-b})\otimes
      \xi)\right] \\
    &= (t_{n-b}^*\otimes t_b^*)(\id_{n-b}\otimes P_n\otimes\id_b)
    (\id_{n-b}\otimes(\id_b\otimes\eta^*)(t_b\otimes\id_{n-b})\otimes
    \xi) t_{n-b} \\
    &= (t_{n-b}^*\otimes t_b^*)(\id_{n-b}\otimes P_n\otimes\id_b)
    ((\id_{n}\otimes\eta^*)t_n \otimes \xi) \\
    & = (t_{n-b}^*\otimes t_b^*)(\id_{n-b}\otimes P_n\otimes\id_b) (\bar\eta
    \otimes \xi).
  \end{align*}
  Since the only intertwiner from $\bar H_n\otimes H_n$ to $\C$, up to a scalar,
  is $\bar\eta\otimes\xi \mapsto t_n^*(\bar\eta\otimes\xi) =
  \langle\xi,\eta\rangle$, this shows that $\Tr(P_{n}(\tilde{g}\otimes f)) = 0$.
  Besides, we have the estimate
  \begin{equation*}
    \|P_{n}(\tilde{g}\otimes f)P_{n}\|_{\HS} \leqslant \|\tilde{g}\otimes f\|_{\HS} = \|\xi\|\|\eta\| = 1
  \end{equation*}
  and similarly $\|P_{k}(g'\otimes \tilde{f}')P_{k}\|_{\HS} \leqslant 1$. Thus,
  Corollary \ref{cor:partialtrace} applies to $F = P_{n}(\tilde{g}\otimes
  f)P_{n} \otimes P_{k}(g'\otimes \tilde{f}')P_{k}$ and yields $\vert Z^{m}\vert
  \leqslant \beta D_{n, k}$.

  \bigskip
  \noindent \textbf{Step 6.}  Now we can rewind the successive approximations to
  bound $S^{m}$. In the remainder of this proof, the symbols $K_{i}$ will denote
  numbers possibly depending on $n$ and $k$, but not on $m$, $l$ and
  $l'$. Recall that $a$, $b$, $c$ are defined in terms of $m$, $l$ and $l'$. To
  bound $T^{m}-Z^{m}$, we use the rough estimate $\vert\Tr_{H}(X)\vert\leqslant
  \dim(H)\|X\|$ which holds for any Hilbert space $H$ and any $X\in\B(H)$. Let
  us note that the operator norms of $f$, $g$, $f'$, $g'$ are dominated by their
  Hilbert-Schmidt norms, which are equal to $1$. However, the space over which
  we take the trace is $H_{1}^{\otimes l}$, which is too big. We therefore take
  advantage of the projections inside the trace to restrict to $H_{b}\otimes
  H_{l'-b}$ and $H_{l-a}\otimes H_{a}$ when passing from $T^{m}$ to $Y^{m}$ and
  to $H_{n}\otimes H_{l'+a-b}$ when passing from $Y^{m}$ to $Z^{m}$. This
  yields:
  \begin{equation*}
    \vert T^{m}\vert \leqslant 
    \vert T^{m} - Y^{m}\vert + \vert Y^{m}-Z^{m}\vert + \vert Z^{m}\vert \leqslant 
    A(d_{b}d_{l'-b} + d_{a}d_{l-a} + \beta d_{n}d_{l'+a-b})q^{c}+\beta D_{n, k}.
  \end{equation*}
  By the second part of Lemma \ref{lem:astuce}, $\beta D_{n, k}$ is bounded by
  $C_{k-a, a}^{-1}C_{n-b, b}^{-1}D_{n, k}$. Because $a\leqslant k$ and
  $b\leqslant n$ take only a finite number of values when $n$ and $k$ are fixed,
  all these constants can be bounded by a constant $K_0$. We can also bound the
  coefficient of $q^{c}$ by
  \begin{equation*}
    A(d_{n}d_{l'}+d_{k}d_{l}+\beta d_{n}d_{l'+k})\leqslant K_{1}q^{-\max(l, l')}.
  \end{equation*}

  Secondly, we have to consider the sum of the $\vert S^{m}_{\mu, \mu'}\vert$'s
  for $(\mu, \mu')\neq (m, m)$. Note that this term is non-zero only if $\mu$
  and $\mu'$ are subrepresentations respectively of $(l-a)\otimes (k-a)$ and
  $(n-b)\otimes (l'-b)$. Thus, there are at most $\min(k-a, l-a)\times\min(n-b,
  l'-b)\leqslant kn$ such terms and each of them is bounded by
  $A\sqrt{d_{l}}\sqrt{d_{l'}}d_ad_bq^{c}$, as explained at the beginning of the
  proof. Also recall from Lemma \ref{lem:kappa} that $\kappa_{m}^{kl}$ and
  $\kappa_{m}^{nl'}$ are respectively bounded by $B_{a}$ and $B_{b}$, and since
  $a$, $b$ take only a finite number of values (determined by $k$ and $n$), they
  are bounded by a constant $K_{2}$. Summing up, we have
  \begin{align*}
    \vert S^{m}\vert & \leqslant K_{2}^{4} \vert T^{m}\vert +  K_{2}^{4}knd_kd_nA\sqrt{d_{l}}\sqrt{d_{l'}}q^{c} \\
    & \leqslant K_{2}^{4} K_{1}q^{c-\max(l,l')} + K_{2}^{4}K_{0} +
    K_{3}q^{c-\max(l,l')}.
  \end{align*}
  Let $t = \min(n+a-b, k+b-a)$. Then, $c\geqslant \max(l, l')-t$ and thus we
  have proved that $\vert S^{m}\vert$ is bounded by a constant $K_{4}$
  independent of $m$, $l$ and $l'$.

  To obtain our estimate for $S$, we now have to sum the $S^{m}$'s. Note that
  for $S^{m}$ to be non-zero, $m = k+l-2a = n+l'-2b$ must be a subrepresentation
  of both $l\otimes k$ and $n\otimes l'$. There are at most $\min(k, n)$ such
  $m$'s and they moreover satisfy $m\geqslant \max(l-k, l'-n)$, so that $d_{m}
  \geqslant K_{5} q^{-\max(l,l')}$ and we can write
  \begin{equation*}
    \vert S\vert \leqslant \sum_{m=0}^{+\infty}\frac{1}{d_{m}}\vert S^{m}\vert\leqslant \min(k, n) K_{5}q^{\max(l,l')}K_{4}.
  \end{equation*}
\end{proof}

\section{The radial subalgebra}\label{sec:radial}

We are now ready to prove the announced results on the radial subalgebra. Before
going into the proofs, we recall the definition of this subalgebra as well as
some of its basic properties.

\begin{de}
  For any finite-dimensional representation $v$ of a compact quantum group $\G$,
  the \emph{character} of $v$ is the element $\chi_{v} = (\id\otimes \Tr)(v)\in
  C(\G)$. This element depends only on the equivalence class of $v$.
  
  The \emph{radial subalgebra} $A \subset L^{\infty}(O_{N}^{+})$ is the von
  Neumann subalgebra generated by the fundamental character $\chi_{1} =
  \chi_{u}$, where $u$ is the matrix of generators.
\end{de}

Note that the radial subalgebra was also used as a sub-C*-algebra $A_{f}$ of the
full C*-algebra $C(O_{N}^{+})$ by M.~Brannan in
\cite{brannan2011approximation}. The spectrum of $\chi_{1}$ in $C(O_{N}^{+})$ is
$\left[-N, N\right]$, whereas it is $[-2,2]$ in $C_{\red}(O_{N}^{+})$ and
$L^{\infty}(O_{N}^{+})$. In the full case, the evaluation functionals
$\mathrm{ev}_{t} : A_{f} \to \C$ at $t \in \left[-N,N\right]$ induce completely
positive maps $T_{t} : L^{\infty}(O_{N}^{+}) \to L^{\infty}(O_{N}^{+})$ which
approximate the identity as $t \to N$. This allowed M.~Brannan to prove that
$L^{\infty}(O_{N}^{+})$ has the Haagerup approximation property.

The terminology is justified by the following analogy with the "classical case"
of the free group factors $\Ll(\F_{N})$. More precisely, denote the standard
generators of $\F_{N}$ by $a_{i}$ and consider
\begin{equation*}
  u = \mathrm{diag}(a_{1}, \dots, a_{N}, a_{1}^{-1}, \dots, a_{N}^{-1})\in\Ll(\F_{N})\otimes \B(\C^{2N}).
\end{equation*}
This is indeed a representation of the compact quantum group dual to $\F_{N}$,
we put $\chi_{1} = \chi_{u} = \sum_{i=1}^{N} (a_{i} + a_{i}^{-1}) \in
\Ll(\F_{N})$ and we define the radial subalgebra $A \subset \Ll(\F_{N})$ as the
von Neumann subalgebra generated by $\chi_{1}$. If we consider, for
$x\in\Ll(\F_{N})$ and $g\in \F_{N}$, the coefficient $x_{g} = \langle x,
g\rangle = \tau(g^{*}x)$ with respect to the standard trace $\tau$, then $x$
belongs to $A$ if and only if the function $(g\mapsto x_{g})$ is \emph{radial},
i.e. $x_{g}$ only depends on the word length of $g$.

The fusion rules of $O_{N}^{+}$ imply that $\chi_{1}\chi_{n} = \chi_{n}\chi_{1}
= \chi_{n+1} + \delta_{n>0}\chi_{n-1}$, so that the radial subalgebra is abelian
and generated as a weakly closed subspace by the characters $(\chi_{n})_{n\in
  \N}$. Moreover, it was proved in \cite{banica1996theorie} that the spectrum of
$\chi_{1}$ in $L^{\infty}(O_{N}^{+})$ is $[-2, 2]$ and that the restriction of
the Haar state is the semi-circle law. More precisely, one can identify $A$ with
$L^{\infty}(\left[-2,2\right])$ via the functional calculus $f \mapsto
f(\chi_{1})$ and the scalar product induced by the Haar state is computed via
\begin{equation*}
  \left\langle f(\chi_{1}), g(\chi_{1})\right\rangle =  \frac{1}{2\pi}\int_{-2}^{2}f(s)\overline{g(s)}\sqrt{4-s^{2}}ds.
\end{equation*}
In particular, the radial subalgebra is diffuse. The characters $\chi_{n}$
correspond to dilated Chebyshev polynomials of the second kind: $\chi_{n}(X) =
T_{n}(X) = U_{n}(X/2)$ where $T_0 = 1$, $T_1 = X$ and $T_{1}T_{n} = T_{n+1} +
T_{n-1}$ if $n\geqslant 1$.

Since $L^{\infty}(O_{N}^{+})$ is a finite von Neumann algebra, there is a unique
$h$-preserving conditional expectation $\E : M\rightarrow A$, which is
explicitly given by
\begin{equation}\label{eq:expectation}
  \E(u^{n}_{\xi \eta}) = \frac{\langle \xi, \eta\rangle}{d_{n}}\chi_{n}.
\end{equation}
We shall denote by $A^{\perp}$ the subspace $\{z\in M, \E(z) = 0\}$, which by
Equation \eqref{eq:expectation} is the weak closure of the linear span of
coefficients $u^{n}_{\xi \eta}$ with $\langle \xi, \eta\rangle = 0$.

As mentioned in the preliminaries, all the results of this article apply in fact
to general free orthogonal quantum groups $O^+(Q)$ \emph{of Kac type}, i.e. such
that $Q$ is a scalar multiple of a unitary matrix. The situation for non-Kac
type free orthogonal quantum groups is however quite different. First recall
that $L^{\infty}(O^+(Q))$ is in that case a type III factor, at least for some
values of the parameter $Q$ (see \cite{vaes2007boundary}). More precisely the
Haar state has then a non-trivial modular group, which is given on the
generating matrix $u \in L^{\infty}(\G)\otimes\B(\C^{N})$ by
\begin{equation*}
  (\sigma_{t}\otimes\id)(u) = (\id\otimes {}^t(Q^{*}Q)^{-it})u(\id\otimes {}^t(Q^{*}Q)^{-it}),
\end{equation*}
where we assume $Q$ to be normalized so that $\Tr(Q^{*}Q) =
\Tr((Q^{*}Q)^{-1})$. In particular, it is clear that $\sigma_{t}(\chi_{1})$ does
not belong to $A$ for all $t$ unless $Q^{*}Q \in\C I_{N}$, and this implies that
there exists no $h$-invariant conditional expectation onto $A$ in the non-Kac
case. It might even be that there exists no normal conditional expectation onto
$A$ at all. On the other hand, as far as we know all the available tools for the
study of abelian subalgebras require the presence of a conditional expectation.

Let us also comment on the $N = 2$ case, where the tools developed in the
previous section break down. If we restrict to Kac type free orthogonal quantum
groups, there are only two examples at $N=2$ up to isomorphism, namely $SU(2)$
and $SU_{-1}(2)$. In the first case $C(SU(2))$ is commutative so that $A$ is
clearly not maximal abelian, and in fact $A$ is not maximal abelian either in
the second case --- this is easily seen by embedding $C(SU_{-1}(2))$ into
$C(S^{3}, M_{2}(\C))$ as in \cite{zakrzewski1991anticommutative}.

With the estimate of Theorem \ref{thm:keyestimate}, we can investigate the
structure of the radial subalgebra. In fact, all the proofs are quite
straightforward using techniques which are well-known to experts in von Neumann
algebras. We however chose to give detailed proof both for convenience of the
reader and for the sake of completeness. From now on, we will write $M =
L^{\infty}(O_{N}^{+})$ and $A = \{\chi_{1}\}''$.

\subsection{Maximal abelianness}

We first prove that $A$ is maximal abelian. This will follow from the following
lemma concerning unitary sequences in $A$, which relies itself on
Theorem~\ref{thm:keyestimate}. In fact here we only use the fact that
$\vert\langle \chi_{l}u_{\xi'\eta'}^{k}, u_{\xi\eta}^{n}\chi_{l'}\rangle\vert
\to 0$ as $l$, $l' \to \infty$ if $\xi$ is orthogonal to $\eta$.

\begin{lem}\label{lem:weakconvergence}
  Let $N\geqslant 3$. Let $(u_{i})_{i}$ be a sequence of unitaries in $A$ weakly
  converging to $0$ and let $z\in A^{\perp}$. Then, $u_{i}zu_{i}^{*}$ converges
  $*$-weakly to $0$.
\end{lem}

\begin{proof}
  For any $i$, let us decompose $u_{i}$ as $u_{i} =
  \sum_{l=0}^{+\infty}a_{l}^{i}\chi_{l}$ and note that by unitarity,
  $\|(a_{l}^{i})_{l}\|_{2} = 1$. Assume for the moment that $z$ is of the form
  $u_{\xi \eta}^{n}$ for some integer $n$ and two orthogonal unit vectors $\xi,
  \eta\in H_{n}$. Considering another integer $k$ and two arbitrary unit vectors
  $\xi', \eta'\in H_{k}$, we will first prove that
  \begin{equation*}
    S_{i} = \vert \langle u_{\xi' \eta'}^{k}, u_{i}u_{\xi \eta}^{n}u_{i}^{*} \rangle\vert = \left\vert\sum_{l, l' = 0}^{+\infty}a_{l}^{i}\overline{a}_{l'}^{i}\langle u_{\xi' \eta'}^{k}, \chi_{l}u_{\xi \eta}^{n} \chi_{l'}\rangle\right\vert \underset{i \rightarrow +\infty}{\longrightarrow} 0.
  \end{equation*}
  Let $\epsilon > 0$ and note that $\langle u_{\xi' \eta'}^{k},
  \chi_{l}u_{\xi\eta}^{n} \chi_{l'}\rangle = \langle \chi_{l}u_{\xi' \eta'}^{k},
  u_{\xi\eta}^{n} \chi_{l'}\rangle$. By Theorem \ref{thm:keyestimate}, there
  exists $L\in \N$ such that $\vert\langle \chi_{l}u_{\xi' \eta'}^{k}, u_{\xi
    \eta}^{n}\chi_{l'}\rangle\vert \leqslant \epsilon/2$ as soon as $l, l' >
  L$. Thus,
  \begin{eqnarray*}
    S_{i} & \leqslant & \sum_{l, l' = 0}^{L}\vert a_{l}^{i}\overline{a}_{l'}^{i}\langle u_{\xi' \eta'}^{k}, \chi_{l}u_{\xi \eta}^{n} \chi_{l'}\rangle\vert + \frac{\epsilon}{2}\sum_{l, l' = L+1}^{+\infty}\vert a_{l}^{i}\overline{a}_{l'}^{i}\vert \\
    & \leqslant & \sum_{l, l' = 0}^{L}\vert a_{l}^{i}\overline{a}_{l'}^{i}\langle u_{\xi' \eta'}^{k}, \chi_{l}u_{\xi \eta}^{n} \chi_{l'}\rangle\vert + \frac{\epsilon}{2}\|(a_{l}^{i})_{l}\|_{2}^{2}
  \end{eqnarray*}
  Now, because $u_{i}\rightarrow 0$ in the weak topology, $a_{l}^{i} =
  h(\chi_{l}u_{i})\rightarrow 0$ for all fixed $l\in \N$ as $i\rightarrow
  +\infty$. In particular, there exists $i_{0}\in \N$ such that for all $i >
  i_{0}$ and all $l, l'\leqslant L$,
  \begin{equation*}
    \vert a_{l}^{i}\overline{a}_{l'}^{i}\vert \leqslant \frac{\epsilon}{2}\left(\sum_{l, l' = 0}^{L}\langle u_{\xi' \eta'}^{k}, \chi_{l}u_{\xi \eta}^{n} \chi_{l'}\rangle\right)^{-1}.
  \end{equation*}
  Thus, for $i>i_{0}$, $\vert \langle u^{k}_{\xi' \eta'},
  u_{i}u^{n}_{\xi\eta}u_{i}^{*}\rangle \vert \leqslant \epsilon$ and $S_{i}
  \rightarrow 0$.

  Making finite linear combinations on the left-hand side, we see that $\langle
  t, u_{i}u^{n}_{\xi \eta}u_{i}^{*}\rangle$ tends to $0$ as $i \to \infty$ for
  any $t\in \Pol(O_{N}^{+})$. Since $\Pol(O_{N}^{+})$ is dense in
  $L^{2}(O_{N}^{+})$ and $(u_{i}u^{n}_{\xi\eta}u_{i}^{*})_{i}$ is bounded
  $L^{2}(O_{N}^{+})$, this is also true for any $t \in L^{\infty}(O_{N}^{+})
  \subset L^{2}(O_{N}^{+})$. Then, we can write $\langle t, u_{i}u^{n}_{\xi
    \eta}u_{i}^{*}\rangle = \langle u_{i}^{*}tu_{i}, u^{n}_{\xi \eta}\rangle$
  and use similarly the density of $A^{\perp}\cap\Pol(O_{N}^{+})$ in $A^{\perp}$
  for the norm of $L^{2}(O_{N}^{+})$. This shows that $\langle t,
  u_{i}zu_{i}^{*}\rangle = \langle u_{i}^{*}tu_{i}, z\rangle \to 0$ as $i\to
  \infty$ for any $t\in M$ and $z\in A^{\perp}$. Since $h$ is a faithful trace
  and $(u_{i}zu_{i}^{*})_{i}$ is bounded in $L^{\infty}(O_{N}^{+})$, this shows
  the stated $*$-weak convergence.
\end{proof}

\begin{thm}\label{thm:maximality}
  Let $N\geqslant 3$. Then, the radial subalgebra $A$ is maximal abelian in $M$.
\end{thm}

\begin{proof}
  Let $x\in A'\cap M$ and consider the decomposition $x = y+z$ with $y\in A$ and
  $z\in A^{\perp}$. Note that
  \begin{equation*}
    x = u_{i}xu_{i}^{*} = u_{i}yu_{i}^{*} + u_{i}zu_{i}^{*} = y + u_{i}zu_{i}^{*},
  \end{equation*}
  so that Lemma \ref{lem:weakconvergence} yields $x = y +
  \lim_{i}u_{i}zu_{i}^{*} = y$.
\end{proof}

The argument above also proves that the C*-algebra generated by $\chi_{1}$ is
maximal abelian in the reduced C*-algebra $C_{\red}(O_{N}^{+})$. From the
theorem, following the strategy of \cite{ricard2005qgaussian}, one can also
recover the factoriality of $L^{\infty}(O_{N}^{+})$ established in
\cite{vaes2007boundary} and also in \cite{fima2014cocycle} (as a byproduct of
non-inner amenability).

\begin{cor}
  For $N\geqslant 3$, the von Neumann algebra $L^{\infty}(O_{N}^{+})$ is a
  factor.
\end{cor}

\begin{proof}
  We exploit the natural action of the classical group $O_{N}$ on $M$ given by
  the following formula, for $g\in O_{N}$ and $x \in C_{\red}(O_{N}^{+})$:
  \begin{equation*}
    \alpha_{g}(x) = (\mathrm{ev}_{g}\pi\otimes\id)\Delta'(x),
  \end{equation*}
  where $\pi : C(O_{N}^{+})\to C(O_{N})$ is the canonical quotient map,
  $\mathrm{ev}_{g} : C(O_{N})\to \C$ is the evaluation map at $g$, and $\Delta'
  : C_{\red}(O_{N}^{+})\to C(O_{N}^{+})\otimes C_{\red}(O_{N}^{+})$ is induced
  from the coproduct of $C(O_{N}^{+})$ thanks to Fell's absorption
  principle. The $*$-automorphism of $C_{\red}(O_{N}^{+})$ defined in this way
  leaves the Haar state $h$ invariant, and thus it extends to $M$. The action of
  $\alpha_{g}$ on coefficients of an irreducible representation $u^{n}$ of
  $O_{N}^{+}$ is given by the following expression, where $v^{n} =
  (\pi\otimes\id)(u^{n})$ is the restriction of $u^{n}$ to $O_{N}$:
  \begin{equation*}
    (\alpha_{g}\otimes\id)(u^{n}) = (\mathrm{ev}_{g}\otimes\id\otimes\id)(v_{13}^{n}u_{23}^{n}) = (1\otimes v^{n}(g))u^{n}.
  \end{equation*}
  In particular we have $\alpha_{g}(\chi_{n}) = \sum_{rs}
  v^{n}(g)_{rs}u^{n}_{sr}$ where $r$, $s$ are indices corresponding to an
  orthonormal basis of $H_{n}$. Note that $\alpha_{g}$ leaves the subspace of
  coefficients of any fixed representation of $O_{N}^{+}$ invariant.

  Since $A$ is maximal abelian in $M$, $\alpha_{g}(A)$ is maximal abelian in $M$
  for every $g \in O_{N}$, and so the center of $M$ is contained in
  $\alpha_{g}(A)$ for every $g\in O_N$. Hence it suffices to show that the
  intersection of the subalgebras $\alpha_{g}(A)$ reduces to $\C
  1$. Equivalently, we take $c \in A$ such that $\alpha_g(c) \in A$ for all $g
  \in O_N$, and we want to prove that $c = \lambda 1$. For this we write $c =
  \sum c_{n}\chi_{n}$ in $L^{2}(O_{N}^{+})$. The orthogonal projection of
  $\alpha_{g}(c)$ onto the subspace generated by the coefficients of $u^{n}$ is
  $c_{n}\alpha_{g}(\chi_{n})$, whereas the projection of $A$ is
  $\C\chi_{n}$. Hence, if $c_{n}\neq 0$ then we must have
  $\alpha_{g}(\chi_{n})\in\C\chi_{n}$ for all $g\in O_{N}$. By the computation
  above and the fact that the coefficients $u^{n}_{rs}$ are linearly
  independent, this happens if and only if $v^{n}(g)$ is scalar for all $g$,
  i.e. $v^{n}$ is a multiple of a one-dimensional representation. But then
  $v^{2n}\subset v^{n}\otimes v^{n}$ would be trivial, and if $n>0$ this would
  imply that $O_{N}$ has only finitely many irreducible representations up to
  equivalence, since any of them is contained in one of the $v^{k}$ and $v^{k+1}
  \subset v^{k}\otimes v^{1}$. Hence $c_{n} = 0$ for all $n>0$.
\end{proof}

\subsection{Singularity and the mixing property}

Now that we know that the radial subalgebra is a MASA, we can investigate
further properties. By \cite{isono2012examples}, we know that $A$ cannot be a
regular MASA (also called Cartan subalgebra) because $M$ is strongly solid. In
view of this result and of the case of the radial MASA in free group factors
treated in \cite{radulescu1991singularity}, it is natural to conjecture that $A$
is \emph{singular}. Recall that for a von Neumann algebra $N$, we denote by
$\mathcal{U}(N)$ the group of unitary elements of $N$.

\begin{de}
  A MASA $A\subset M$ is said to be singular if $\{u\in \mathcal{U}(M),
  uAu^{*}\subset A\} = \mathcal{U}(A)$.
\end{de}

There are several ways of proving that a MASA is singular. One way goes through
a von Neumann algebraic analogue of the mixing property for group actions,
called weak mixing, which eventually turns out to be equivalent to
singularity. In our case, we can prove a stronger statement than singularity,
namely that $A$ is \emph{mixing} in the following sense:

\begin{de}
  A subalgebra $A$ of a von Neumann algebra $M$ is said to be mixing if for any
  sequence $(u_{n})_{n}$ of unitaries in $A$ converging weakly to $0$ and any
  elements $x, y\in A^{\perp}$,
  \begin{equation*}
    \|\E_{A}(xu_{n}y)\|_{2} \longrightarrow 0.
  \end{equation*}
\end{de}

Again, the proof is an easy application of Theorem \ref{thm:keyestimate}.

\begin{thm}\label{thm:mixing}
  For $N \geqslant 3$ the radial MASA is mixing.
\end{thm}

\begin{proof}
  Fix a sequence of unitaries $u_i \in A$ converging weakly to $0$.  Let $k,
  n\in \N$ and consider two pairs of orthogonal unit vectors $\xi, \eta\in
  H_{n}$ and $\xi', \eta'\in H_{k}$. Since elements of the form
  $u^{k}_{\xi'\eta'}$ (resp. $u^{n*}_{\xi \eta}$) with $\xi'\perp\eta'$
  (resp. $\xi\perp\eta$) span a dense subspace of $A^\perp \subset
  L^{2}(O_{N}^{+})$, it is enough to prove that $X_i\to 0$, where
  \begin{equation*}
    X_i = \|\E(u^{n*}_{\xi \eta}u_{i}u^{k}_{\xi' \eta'})\|_{2}^{2}.
  \end{equation*}
  To compute the square norm, we can use the orthonormal basis given by the
  characters to get
  \begin{eqnarray*}
    X_i & = & \sum_{l'=0}^{+\infty}\vert \langle\E(u^{n*}_{\xi \eta}u_{i}u^{k}_{\xi' \eta'}), \chi_{l'} \rangle\vert^{2} \\
    & = & \sum_{l'=0}^{+\infty}\vert \langle u^{n*}_{\xi \eta}u_{i}u^{k}_{\xi' \eta'}, \chi_{l'} \rangle\vert^{2} \\
    & = & \sum_{l'=0}^{+\infty}\vert \langle u_{i}u^{k}_{\xi' \eta'}, u^{n}_{\xi \eta}\chi_{l'} \rangle\vert^{2}.
  \end{eqnarray*}
  Since $u_i$ converges weakly to $0$, each term of the sum above tends to $0$
  as $i \to \infty$. Hence it suffices to show that the dominated convergence
  theorem applies. For this we decompose the unitaries $u_{i}$ according to the
  basis of characters: $u_{i} = \sum_{i=0}^{+\infty}a_{l}^{i}\chi_{l}$, with
  $\sum_l |a^i_l|^2 = \|u_i\|_2^2 = 1$. Then, the Cauchy-Schwartz inequality and
  Theorem \ref{thm:keyestimate} yield
  \begin{eqnarray*}
    X_i & = & \sum_{l'=0}^{+\infty}\left\vert\sum_{l=0}^{+\infty}
      a_{l}^{i} \langle \chi_{l}u^{k}_{\xi'\eta'}, 
      u^{n}_{\xi \eta}\chi_{l'} \rangle\right\vert^{2} \\
    & \leqslant & \sum_{l'=0}^{+\infty}\sum_{l=0}^{+\infty}
    \vert\langle \chi_{l}u^{k}_{\xi'\eta'}, 
    u^{n}_{\xi \eta}\chi_{l'} \rangle\vert^{2} \\ & \leqslant &
    K\sum_{l'=0}^{+\infty}\sum_{l=0}^{+\infty}q^{\max(l, l')} < +\infty.
  \end{eqnarray*}
\end{proof}

\begin{cor}
  The radial MASA is singular.
\end{cor}

\begin{proof}
  Since $A$ is diffuse, we can find a sequence of unitaries $u_n \in A$
  converging weakly to $0$. Let $v \in \mathcal{U}(M)$ be a unitary such that
  $vAv^* \subset A$. In particular we have $vu_nv^* \in A$, hence
  $\|\E_A(u_n^*vu_nv^*)\|_2 = \|u_n^*vu_nv^*\|_2 = 1$. On the other hand the
  mixingness property implies that $\|\E_A(u_n^*vu_nv^*) - \E_A(v) \E_A(v^*)\|_2
  \to 0$. As a result we have $1 = \|\E_A(v) \E_A(v^*)\|_2 \leq \|\E_A(v)\|_2
  \leq \|v\|_2 = 1$, hence $\E_A(v)=v$ and $v \in A$.
\end{proof}

\subsection{The spectral measure}

Another very natural problem for a given MASA is to study the $A$-$A$-bimodule
structure of $H = L^{2}(M)\ominus L^{2}(A)$. This can be done through the
associated \emph{spectral measure}. Because the representations of $A$ on $H$ on
the left and on the right commute, their images generate an abelian von Neumann
subalgebra of $\B(H)$ isomorphic to $L^{\infty}([-2, 2]\times [-2, 2])$. Thus,
disintegrating $H$ with respect to this subalgebra yields a measure class
$[\nu]$ on $[-2, 2]\times [-2, 2]$ which encapsulates some properties of the
bimodule.

We first recall an elementary lemma about the Chebyshev polynomials $U_n$, which
are linked to the characters $\chi_n$ by $\chi_n = T_n(\chi_1) = U_n(\chi_1/2)$:

\begin{lem} \label{lem:chebyshevestimate} For all $n \in \N$ we have $\sup_{t\in
    \left[-1,1\right]} | U_n(t) | = n+1$. Moreover, for all $r > 1$ there exists
  an open neighborhood $\Omega$ of $\left[-1,1\right]$ in $\C$ such that
  $\sup_{z\in \Omega} |U_n(z)| \leqslant (n+1)r^{n}$ for all $n$.
\end{lem}

\begin{proof}
  The first assertion is well known and follows immediately from the formula
  \begin{displaymath}
    U_n(\cos\theta) = \frac{\sin((n+1)\theta)}{\sin(\theta)} 
    = \sum_{k=0}^n e^{i(n-2k)\theta}.
  \end{displaymath}
  The second assertion is probably also known. Let us denote $\varphi(w) = \frac
  12(w+w^{-1})$ for $w \in \C^*$.  By the identity theorem we have
  $U_n(\varphi(w)) = \sum_{k=0}^n w^{n-2k}$ hence $|U_n (\varphi(w))| \leqslant
  (n+1)|w|^{n}$ if $|w|\geqslant 1$.

  On the other hand one can compute $\varphi(se^{i\theta}) = \frac
  12(s+s^{-1})\cos(\theta) + \frac i2(s-s^{-1})\sin(\theta)$, hence for
  $s\geqslant 1$ the image by $\varphi$ of the circle $C_s = \{w\in\C \mid
  |w|=s\}$ is the ellipse $E_s$ with axes $\frac
  12\left[-s-s^{-1},s+s^{-1}\right]$ and $\frac
  i2\left[-s+s^{-1},s-s^{-1}\right]$. Note that when $s$ decreases to $1$,
  $\frac 12 (s+s^{-1})$ (resp. $\frac 12(s-s^{-1})$) decreases to $1$
  (resp. $0$). In particular one sees that for $r > 1$ the image $\Omega =
  \varphi(D_r)$ of $D_r = \{w\in\C \mid 1\leqslant |w| < r\}$ is an open
  neighborhood of $\left[-1,1\right]$. For $z \in \Omega$ one obtains then
  $|U_n(z)| < (n+1) r^{n}$ by writing $z = \varphi(w)$ with $w \in D_r$.
\end{proof}

\begin{thm}
  For $N\geqslant 3$ the measure $\nu$ is Lebesgue equivalent to $\lambda\otimes
  \lambda$, where $\lambda$ denotes the Lebesgue measure on $[-2, 2]$.
\end{thm}

\begin{proof}
  We will follow the strategy of \cite{dykema2013measure}. Let us first look at
  some "projections" of $\nu$ in the following sense: for two integers $k$ and
  $n$ and two pairs of orthogonal unit vectors $\xi, \eta\in H_{n}$ and $\xi',
  \eta'\in H_{k}$, there exists a complex measure $\mu$ on $[-2, 2]\times [-2,
  2]$ such that for any $a, b\in A$,
  \begin{equation*}
    \langle au^{k}_{\xi'\eta'}b, u^{n}_{\xi, \eta} \rangle = \iint_{[-2, 2]\times [-2, 2]}a(s)b(t)d\mu(s, t).
  \end{equation*}
  We will compute the Radon-Nikodym derivative of $\mu$ with respect to
  $\lambda\otimes \lambda$. To do this, let us set
  \begin{align*}
    D_{l, l'} &= \langle \chi_{l}u^{k}_{\xi'\eta'}\chi_{l'}, u^{n}_{\xi, \eta} \rangle \\
    f(z,z') &= \sum_{l,l'=0}^{\infty} A_{l,l'}(z,z') \text{~~~where~~~}
    A_{l,l'}(z,z') = T_l(z)T_{l'}(z')D_{l,l'}.
  \end{align*}
  Now we fix $r \in \left]1, 1/\sqrt q\right[$ and we consider the open subset
  $\left]-1,1\right[ \subset \Omega \subset \C$ given by
  Lemma~\ref{lem:chebyshevestimate}. Theorem~\ref{thm:keyestimate} yields the
  following estimate for the supremum norm on $2\Omega \times 2\Omega$:
  \begin{align*}
    \|A_{l,l'}\|_{\infty} &\leqslant K \|T_l\|_\infty \|T_{l'}\|_\infty
    q^{\max(l,l')} \leqslant K (l+1) r^l q^{l/2} (l'+1) r^{l'} q^{l'/2}.
  \end{align*}
  This implies that the series of functions defining $f$ converges normally on
  $2\Omega \times 2\Omega$. Since all summands are polynomials $f$ is
  holomorphic on $2\Omega \times 2\Omega$ and in particular analytic on
  $\left[-2,2\right] \times \left[-2,2\right]$.

  The function $f$ is linked to the measure $\mu$ by the following computation:
  \begin{align*}
    \langle au^{k}_{\xi'\eta'}b, u^{n}_{\xi, \eta} \rangle & = \sum_{l,
      l'=0}^{+\infty}\langle a, \chi_{l}\rangle \langle b, \chi_{l'}\rangle
    \langle \chi_{l}u^{k}_{\xi'\eta'}\chi_{l'}, u^{n}_{\xi, \eta} \rangle =
    \sum_{l, l'=0}^{+\infty}\langle a, \chi_{l}\rangle
    \langle b, \chi_{l'}\rangle D_{l, l'} \\
    & = \frac 1{4\pi^2} \sum_{l, l'=0}^{+\infty} D_{l, l'}
    \left(\int_{-2}^{2}a(s)T_{l}(s)\sqrt{4-s^{2}}ds\right)
    \left(\int_{-2}^{2}b(t)T_{l}(t)\sqrt{4-t^{2}}dt\right)  \\
    & = \frac 1{4\pi^2} \iint_{[-2, 2]\times [-2, 2]}a(s)b(t)f(s, t)
    \sqrt{4-s^{2}}\sqrt{4-t^{2}}d(\lambda\otimes\lambda)(s,t).
  \end{align*}
  Hence, $f(s, t) \sqrt{4-s^2}\sqrt{4-t^2}$ is the Radon-Nikodym derivative of
  $\mu$ with respect to $\lambda\otimes \lambda$.

  Consider now an arbitrary element $\zeta$ in $\Pol(O_{N}^{+})\cap
  A^{\perp}$. It can be written as a finite linear combination of coefficients
  corresponding to orthogonal vectors as above, hence the probability measure
  $\mu_{\zeta}$ defined by
  \begin{equation*}
    \langle a\zeta b, \zeta \rangle = \iint_{[-2, 2]\times [-2, 2]}a(s)b(t)d\mu_{\zeta}(s, t).
  \end{equation*}
  has a density of the form $f(s, t) \sqrt{4-s^2}\sqrt{4-t^2}$ with respect to
  $\lambda\otimes \lambda$, where $f$ is analytic on $\left[-2,2\right] \times
  \left[-2,2\right]$. Since $\mu_\zeta$ is obviously non-zero, $f$ does not
  vanish identically and by analyticity its zeros are contained in a set of
  Lebesgue measure $0$, so that $\mu_\zeta$ is equivalent to
  $\lambda\otimes\lambda$.  Because $\Pol(O_{N}^{+})\cap A^{\perp}$ is dense in
  $L^{2}(M)\ominus L^{2}(A)$, this implies that $[\nu] = [\lambda\otimes
  \lambda]$.
\end{proof}

Note that as a consequence, the $A-A$-bimodule $L^{2}(M)\ominus L^{2}(A)$ is
contained in a multiple of the coarse bimodule, see
\cite[Section~2]{mukherjee2013measure}. Since the coarse bimodule is mixing, we
can also recover Theorem \ref{thm:mixing} in this way.

\subsection{Concluding remarks}

We would like to briefly discuss some possible extensions of this work. First
consider the quantum automorphism group $\G(M_{N}(\C), \tr)$ of $M_{N}(\C)$
endowed with the canonical trace. It is known that the von Neumann algebra
$L^{\infty}(\G(M_{N}(\C), \tr))$ of this quantum group embeds into
$L^{\infty}(O_{N}^{+})$ as the subalgebra generated by all $u_{\xi, \eta}^{2n}$
for $n\in \N$ and $\xi, \eta\in H_{2n}$. Let us set $v^{n} = u^{2n}$. Then, the
$v^{n}$'s form a complete family of representatives of irreducible
representations of $\G(M_{N}(\C), \tr)$ with corresponding characters $\psi_{n}
= \chi_{2n}$. In particular, for any orthogonal unit vectors $\xi, \eta\in
H_{2n}$ and $\xi', \eta'\in H_{2k}$,
\begin{equation*}
  \langle \psi_{l}v^{k}_{\xi', \eta'}, v^{n}_{\xi, \eta}\psi_{l'}\rangle \leqslant K q^{\max(2l, 2l')}
\end{equation*}
by Theorem \ref{thm:keyestimate}. From this we see that the radial subalgebra in
$L^{\infty}(\G(M_{N}(\C), \tr))$ is maximal abelian and mixing and that its
associated bimodule is a direct sum of coarse bimodules. This is an interesting
example because $\G(M_{N}(\C), \tr)$ has $SO(3)$-type fusion rules, like another
important family of discrete quantum groups called the quantum permutation
groups $S_{N}^{+}$. This of course suggests that our result extends to
$S_{N}^{+}$. One way to prove this may be through monoidal equivalence
\cite{bichon2006ergodic}.

Another possible extension of our work would be to the non-Kac case. It is
possible that the estimate of Theorem \ref{thm:keyestimate} still holds with
appropriate modification for an arbitrary free orthogonal quantum
group. However, the proofs of Section \ref{sec:radial} all break down if the von
Neumann algebra is type III, because the radial MASA has no $h$-invariant
conditional expectation in that case. There is therefore an additional von
Neumann algebraic problem to solve in that case, but this could yield very
explicit examples of singular MASAs in type III factors.

\bibliographystyle{amsplain} \bibliography{MASA_FOn}

\end{document}

%% file: rotation_def.latex
\setlength{\unitlength}{4144sp}%
\begingroup\makeatletter\ifx\SetFigFont\undefined%
\gdef\SetFigFont#1#2#3#4#5{%
  \reset@font\fontsize{#1}{#2pt}%
  \fontfamily{#3}\fontseries{#4}\fontshape{#5}%
  \selectfont}%
\fi\endgroup%
\begin{picture}(1107,1374)(2236,-973)
\thinlines
{\color[rgb]{0,0,0}\put(3151,-421){\oval(180,180)[bl]}
\put(3151,-421){\oval(180,180)[br]}
}%
{\color[rgb]{0,0,0}\put(2521,-151){\oval(180,180)[tr]}
\put(2521,-151){\oval(180,180)[tl]}
}%
{\color[rgb]{0,0,0}\put(3061,299){\circle*{10}}
}%
{\color[rgb]{0,0,0}\put(3016,299){\circle*{10}}
}%
{\color[rgb]{0,0,0}\put(2971,299){\circle*{10}}
}%
{\color[rgb]{0,0,0}\put(2926,-61){\circle*{10}}
}%
{\color[rgb]{0,0,0}\put(2881,-61){\circle*{10}}
}%
{\color[rgb]{0,0,0}\put(2836,-61){\circle*{10}}
}%
{\color[rgb]{0,0,0}\put(2836,-511){\circle*{10}}
}%
{\color[rgb]{0,0,0}\put(2791,-511){\circle*{10}}
}%
{\color[rgb]{0,0,0}\put(2746,-511){\circle*{10}}
}%
{\color[rgb]{0,0,0}\put(2791,-871){\circle*{10}}
}%
{\color[rgb]{0,0,0}\put(2746,-871){\circle*{10}}
}%
{\color[rgb]{0,0,0}\put(2701,-871){\circle*{10}}
}%
{\color[rgb]{0,0,0}\put(3241,-421){\line( 0, 1){450}}
}%
{\color[rgb]{0,0,0}\put(2431,-151){\line( 0,-1){450}}
}%
{\color[rgb]{0,0,0}\put(2611,-421){\line( 0,-1){180}}
}%
{\color[rgb]{0,0,0}\put(2971,-421){\line( 0,-1){180}}
}%
{\color[rgb]{0,0,0}\put(3061,-151){\line( 0, 1){180}}
}%
{\color[rgb]{0,0,0}\put(2701,-151){\line( 0, 1){180}}
}%
{\color[rgb]{0,0,0}\put(2431,-781){\line( 0,-1){180}}
}%
{\color[rgb]{0,0,0}\put(2971,-781){\line( 0,-1){180}}
}%
{\color[rgb]{0,0,0}\put(2521,-781){\line( 0,-1){180}}
}%
{\color[rgb]{0,0,0}\put(2701,209){\line( 0, 1){180}}
}%
{\color[rgb]{0,0,0}\put(3241,209){\line( 0, 1){180}}
}%
{\color[rgb]{0,0,0}\put(2791,209){\line( 0, 1){180}}
}%
{\color[rgb]{0,0,0}\put(2341,-781){\framebox(720,180){}}
}%
{\color[rgb]{0,0,0}\put(2521,-421){\framebox(630,270){}}
}%
{\color[rgb]{0,0,0}\put(2611, 29){\framebox(720,180){}}
}%
\put(2971, 96){\makebox(0,0)[b]{\smash{{\SetFigFont{12}{14.4}{\rmdefault}{\mddefault}{\updefault}{\color[rgb]{0,0,0}$p_k$}%
}}}}
\put(2701,-714){\makebox(0,0)[b]{\smash{{\SetFigFont{12}{14.4}{\rmdefault}{\mddefault}{\updefault}{\color[rgb]{0,0,0}$p_k$}%
}}}}
\put(2251,-376){\makebox(0,0)[rb]{\smash{{\SetFigFont{12}{14.4}{\rmdefault}{\mddefault}{\updefault}{\color[rgb]{0,0,0}$\rho(f)=$}%
}}}}
\put(2836,-331){\makebox(0,0)[b]{\smash{{\SetFigFont{12}{14.4}{\rmdefault}{\mddefault}{\updefault}{\color[rgb]{0,0,0}$f$}%
}}}}
\end{picture}%

%% file: rotation_trace_init.latex
\setlength{\unitlength}{4144sp}%
\begingroup\makeatletter\ifx\SetFigFont\undefined%
\gdef\SetFigFont#1#2#3#4#5{%
  \reset@font\fontsize{#1}{#2pt}%
  \fontfamily{#3}\fontseries{#4}\fontshape{#5}%
  \selectfont}%
\fi\endgroup%
\begin{picture}(1290,564)(2146,-1873)
\thinlines
{\color[rgb]{0,0,0}\put(2521,-1501){\oval(180,180)[tr]}
\put(2521,-1501){\oval(180,180)[tl]}
}%
{\color[rgb]{0,0,0}\put(2701,-1681){\oval(180,180)[bl]}
\put(2701,-1681){\oval(180,180)[br]}
}%
{\color[rgb]{0,0,0}\put(2521,-1681){\framebox(180,180){}}
}%
{\color[rgb]{0,0,0}\put(2791,-1681){\line( 0, 1){180}}
}%
{\color[rgb]{0,0,0}\put(2431,-1681){\line( 0, 1){180}}
}%
{\color[rgb]{0,0,0}\multiput(2791,-1501)(0.00000,120.00000){2}{\line( 0, 1){ 60.000}}
\multiput(2791,-1321)(-128.57143,0.00000){4}{\line(-1, 0){ 64.286}}
\multiput(2341,-1321)(0.00000,-120.00000){5}{\line( 0,-1){ 60.000}}
\put(2341,-1861){\line( 1, 0){ 90}}
\multiput(2431,-1861)(0.00000,120.00000){2}{\line( 0, 1){ 60.000}}
}%
{\color[rgb]{0,0,0}\put(3061,-1681){\framebox(180,180){}}
}%
{\color[rgb]{0,0,0}\put(3151,-1501){\line( 0, 1){ 90}}
\multiput(3151,-1411)(120.00000,0.00000){2}{\line( 1, 0){ 60.000}}
\multiput(3331,-1411)(0.00000,-102.85714){4}{\line( 0,-1){ 51.429}}
\multiput(3331,-1771)(-120.00000,0.00000){2}{\line(-1, 0){ 60.000}}
\put(3151,-1771){\line( 0, 1){ 90}}
}%
\put(2611,-1636){\makebox(0,0)[b]{\smash{{\SetFigFont{12}{14.4}{\rmdefault}{\mddefault}{\updefault}{\color[rgb]{0,0,0}$f$}%
}}}}
\put(2926,-1636){\makebox(0,0)[b]{\smash{{\SetFigFont{12}{14.4}{\rmdefault}{\mddefault}{\updefault}{\color[rgb]{0,0,0}$=$}%
}}}}
\put(2161,-1636){\makebox(0,0)[rb]{\smash{{\SetFigFont{12}{14.4}{\rmdefault}{\mddefault}{\updefault}{\color[rgb]{0,0,0}$\Tr(\rho(f))=$}%
}}}}
\put(3421,-1636){\makebox(0,0)[lb]{\smash{{\SetFigFont{12}{14.4}{\rmdefault}{\mddefault}{\updefault}{\color[rgb]{0,0,0}$=\Tr(f)$}%
}}}}
\put(3151,-1636){\makebox(0,0)[b]{\smash{{\SetFigFont{12}{14.4}{\rmdefault}{\mddefault}{\updefault}{\color[rgb]{0,0,0}$f$}%
}}}}
\end{picture}%

%% file: rotation_trace_trans.latex
\setlength{\unitlength}{4144sp}%
\begingroup\makeatletter\ifx\SetFigFont\undefined%
\gdef\SetFigFont#1#2#3#4#5{%
  \reset@font\fontsize{#1}{#2pt}%
  \fontfamily{#3}\fontseries{#4}\fontshape{#5}%
  \selectfont}%
\fi\endgroup%
\begin{picture}(2862,1104)(1066,-1738)
{\color[rgb]{0,0,0}\thinlines
\put(1846,-781){\circle*{10}}
}%
{\color[rgb]{0,0,0}\put(1891,-781){\circle*{10}}
}%
{\color[rgb]{0,0,0}\put(1936,-781){\circle*{10}}
}%
{\color[rgb]{0,0,0}\put(1756,-1231){\circle*{10}}
}%
{\color[rgb]{0,0,0}\put(1801,-1231){\circle*{10}}
}%
{\color[rgb]{0,0,0}\put(1846,-1231){\circle*{10}}
}%
{\color[rgb]{0,0,0}\put(1621,-1591){\circle*{10}}
}%
{\color[rgb]{0,0,0}\put(1666,-1591){\circle*{10}}
}%
{\color[rgb]{0,0,0}\put(1711,-1591){\circle*{10}}
}%
{\color[rgb]{0,0,0}\put(3601,-781){\circle*{10}}
}%
{\color[rgb]{0,0,0}\put(3646,-781){\circle*{10}}
}%
{\color[rgb]{0,0,0}\put(3691,-781){\circle*{10}}
}%
{\color[rgb]{0,0,0}\put(3331,-1591){\circle*{10}}
}%
{\color[rgb]{0,0,0}\put(3376,-1591){\circle*{10}}
}%
{\color[rgb]{0,0,0}\put(3421,-1591){\circle*{10}}
}%
{\color[rgb]{0,0,0}\put(3466,-1231){\circle*{10}}
}%
{\color[rgb]{0,0,0}\put(3511,-1231){\circle*{10}}
}%
{\color[rgb]{0,0,0}\put(3556,-1231){\circle*{10}}
}%
{\color[rgb]{0,0,0}\put(1531,-871){\oval(180,180)[tr]}
\put(1531,-871){\oval(180,180)[tl]}
}%
{\color[rgb]{0,0,0}\put(2161,-1141){\oval(180,180)[bl]}
\put(2161,-1141){\oval(180,180)[br]}
}%
{\color[rgb]{0,0,0}\put(3286,-871){\oval(180,180)[tr]}
\put(3286,-871){\oval(180,180)[tl]}
}%
{\color[rgb]{0,0,0}\put(3736,-1501){\oval(180,180)[bl]}
\put(3736,-1501){\oval(180,180)[br]}
}%
{\color[rgb]{0,0,0}\put(1441,-871){\line( 0,-1){450}}
}%
{\color[rgb]{0,0,0}\put(2251,-1141){\line( 0, 1){360}}
}%
{\color[rgb]{0,0,0}\put(2071,-871){\line( 0, 1){ 90}}
}%
{\color[rgb]{0,0,0}\put(1711,-871){\line( 0, 1){180}}
}%
{\color[rgb]{0,0,0}\put(1621,-1141){\line( 0,-1){180}}
}%
{\color[rgb]{0,0,0}\put(1981,-1141){\line( 0,-1){180}}
}%
{\color[rgb]{0,0,0}\put(1981,-1501){\line( 0,-1){ 90}}
}%
{\color[rgb]{0,0,0}\put(1441,-1501){\line( 0,-1){ 90}}
}%
{\color[rgb]{0,0,0}\put(1891,-1501){\line( 0,-1){ 90}}
}%
{\color[rgb]{0,0,0}\put(1531,-1141){\framebox(630,270){}}
}%
{\color[rgb]{0,0,0}\put(1351,-1501){\framebox(720,180){}}
}%
{\color[rgb]{0,0,0}\multiput(1711,-691)(-128.57143,0.00000){4}{\line(-1, 0){ 64.286}}
\multiput(1261,-691)(0.00000,-116.47059){9}{\line( 0,-1){ 58.235}}
\multiput(1261,-1681)(120.00000,0.00000){2}{\line( 1, 0){ 60.000}}
\put(1441,-1681){\line( 0, 1){ 90}}
}%
{\color[rgb]{0,0,0}\multiput(2071,-781)(0.00000,90.00000){2}{\line( 0, 1){ 45.000}}
\multiput(2071,-646)(-114.00000,0.00000){8}{\line(-1, 0){ 57.000}}
\multiput(1216,-646)(0.00000,-113.68421){10}{\line( 0,-1){ 56.842}}
\multiput(1216,-1726)(122.72727,0.00000){6}{\line( 1, 0){ 61.364}}
\multiput(1891,-1726)(0.00000,90.00000){2}{\line( 0, 1){ 45.000}}
}%
{\color[rgb]{0,0,0}\multiput(2251,-781)(0.00000,90.00000){2}{\line( 0, 1){ 45.000}}
\put(2251,-646){\line( 1, 0){ 90}}
\multiput(2341,-646)(0.00000,-113.68421){10}{\line( 0,-1){ 56.842}}
\multiput(2341,-1726)(-102.85714,0.00000){4}{\line(-1, 0){ 51.429}}
\multiput(1981,-1726)(0.00000,90.00000){2}{\line( 0, 1){ 45.000}}
}%
{\color[rgb]{0,0,0}\put(3196,-871){\line( 0,-1){450}}
}%
{\color[rgb]{0,0,0}\put(3826,-871){\line( 0, 1){ 90}}
}%
{\color[rgb]{0,0,0}\put(3466,-871){\line( 0, 1){180}}
}%
{\color[rgb]{0,0,0}\put(3376,-1141){\line( 0,-1){180}}
}%
{\color[rgb]{0,0,0}\put(3196,-1501){\line( 0,-1){ 90}}
}%
{\color[rgb]{0,0,0}\put(3286,-1141){\framebox(630,270){}}
}%
{\color[rgb]{0,0,0}\multiput(3466,-691)(-128.57143,0.00000){4}{\line(-1, 0){ 64.286}}
\multiput(3016,-691)(0.00000,-116.47059){9}{\line( 0,-1){ 58.235}}
\multiput(3016,-1681)(120.00000,0.00000){2}{\line( 1, 0){ 60.000}}
\put(3196,-1681){\line( 0, 1){ 90}}
}%
{\color[rgb]{0,0,0}\multiput(3826,-781)(0.00000,90.00000){2}{\line( 0, 1){ 45.000}}
\multiput(3826,-646)(-114.00000,0.00000){8}{\line(-1, 0){ 57.000}}
\multiput(2971,-646)(0.00000,-113.68421){10}{\line( 0,-1){ 56.842}}
\multiput(2971,-1726)(106.36364,0.00000){6}{\line( 1, 0){ 53.182}}
\multiput(3556,-1726)(0.00000,90.00000){2}{\line( 0, 1){ 45.000}}
}%
{\color[rgb]{0,0,0}\put(3556,-1501){\line( 0,-1){ 90}}
}%
{\color[rgb]{0,0,0}\put(3646,-1141){\line( 0,-1){180}}
}%
{\color[rgb]{0,0,0}\put(3826,-1501){\line( 0, 1){360}}
}%
{\color[rgb]{0,0,0}\put(3106,-1501){\framebox(630,180){}}
}%
\put(1081,-1276){\makebox(0,0)[rb]{\smash{{\SetFigFont{12}{14.4}{\rmdefault}{\mddefault}{\updefault}{\color[rgb]{0,0,0}$\Tr(\rho(f))=$}%
}}}}
\put(2656,-1276){\makebox(0,0)[b]{\smash{{\SetFigFont{12}{14.4}{\rmdefault}{\mddefault}{\updefault}{\color[rgb]{0,0,0}$=$}%
}}}}
\put(1846,-1051){\makebox(0,0)[b]{\smash{{\SetFigFont{12}{14.4}{\rmdefault}{\mddefault}{\updefault}{\color[rgb]{0,0,0}$f$}%
}}}}
\put(3601,-1051){\makebox(0,0)[b]{\smash{{\SetFigFont{12}{14.4}{\rmdefault}{\mddefault}{\updefault}{\color[rgb]{0,0,0}$f$}%
}}}}
\put(1711,-1434){\makebox(0,0)[b]{\smash{{\SetFigFont{12}{14.4}{\rmdefault}{\mddefault}{\updefault}{\color[rgb]{0,0,0}$p_k$}%
}}}}
\put(3421,-1434){\makebox(0,0)[b]{\smash{{\SetFigFont{12}{14.4}{\rmdefault}{\mddefault}{\updefault}{\color[rgb]{0,0,0}$p_k$}%
}}}}
\end{picture}%

%% file: rotation_trace_last.latex
\setlength{\unitlength}{4144sp}%
\begingroup\makeatletter\ifx\SetFigFont\undefined%
\gdef\SetFigFont#1#2#3#4#5{%
  \reset@font\fontsize{#1}{#2pt}%
  \fontfamily{#3}\fontseries{#4}\fontshape{#5}%
  \selectfont}%
\fi\endgroup%
\begin{picture}(2592,1284)(2734,-1918)
{\color[rgb]{0,0,0}\thinlines
\put(3376,-781){\circle*{10}}
}%
{\color[rgb]{0,0,0}\put(3421,-781){\circle*{10}}
}%
{\color[rgb]{0,0,0}\put(3466,-781){\circle*{10}}
}%
{\color[rgb]{0,0,0}\put(3106,-1771){\circle*{10}}
}%
{\color[rgb]{0,0,0}\put(3151,-1771){\circle*{10}}
}%
{\color[rgb]{0,0,0}\put(3196,-1771){\circle*{10}}
}%
{\color[rgb]{0,0,0}\put(3196,-1321){\circle*{10}}
}%
{\color[rgb]{0,0,0}\put(3241,-1321){\circle*{10}}
}%
{\color[rgb]{0,0,0}\put(3286,-1321){\circle*{10}}
}%
{\color[rgb]{0,0,0}\put(4816,-781){\circle*{10}}
}%
{\color[rgb]{0,0,0}\put(4861,-781){\circle*{10}}
}%
{\color[rgb]{0,0,0}\put(4906,-781){\circle*{10}}
}%
{\color[rgb]{0,0,0}\put(4816,-1591){\circle*{10}}
}%
{\color[rgb]{0,0,0}\put(4861,-1591){\circle*{10}}
}%
{\color[rgb]{0,0,0}\put(4906,-1591){\circle*{10}}
}%
{\color[rgb]{0,0,0}\put(4771,-1231){\circle*{10}}
}%
{\color[rgb]{0,0,0}\put(4816,-1231){\circle*{10}}
}%
{\color[rgb]{0,0,0}\put(4861,-1231){\circle*{10}}
}%
{\color[rgb]{0,0,0}\put(3061,-871){\oval(180,180)[tr]}
\put(3061,-871){\oval(180,180)[tl]}
}%
{\color[rgb]{0,0,0}\put(3421,-1501){\oval(180,180)[tr]}
\put(3421,-1501){\oval(180,180)[tl]}
}%
{\color[rgb]{0,0,0}\put(3061,-1141){\oval(180,180)[bl]}
\put(3061,-1141){\oval(180,180)[br]}
}%
{\color[rgb]{0,0,0}\put(3601,-1681){\oval(180,180)[bl]}
\put(3601,-1681){\oval(180,180)[br]}
}%
{\color[rgb]{0,0,0}\put(4411,-871){\oval(180,180)[tr]}
\put(4411,-871){\oval(180,180)[tl]}
}%
{\color[rgb]{0,0,0}\put(4411,-1321){\oval(180,180)[bl]}
\put(4411,-1321){\oval(180,180)[br]}
}%
{\color[rgb]{0,0,0}\put(3601,-871){\line( 0, 1){ 90}}
}%
{\color[rgb]{0,0,0}\put(3241,-871){\line( 0, 1){180}}
}%
{\color[rgb]{0,0,0}\multiput(3241,-691)(-128.57143,0.00000){4}{\line(-1, 0){ 64.286}}
\multiput(2791,-691)(0.00000,-111.42857){11}{\line( 0,-1){ 55.714}}
\multiput(2791,-1861)(120.00000,0.00000){2}{\line( 1, 0){ 60.000}}
\put(2971,-1861){\line( 0, 1){ 90}}
}%
{\color[rgb]{0,0,0}\multiput(3601,-781)(0.00000,90.00000){2}{\line( 0, 1){ 45.000}}
\multiput(3601,-646)(-114.00000,0.00000){8}{\line(-1, 0){ 57.000}}
\multiput(2746,-646)(0.00000,-109.56522){12}{\line( 0,-1){ 54.783}}
\multiput(2746,-1906)(106.36364,0.00000){6}{\line( 1, 0){ 53.182}}
\multiput(3331,-1906)(0.00000,90.00000){2}{\line( 0, 1){ 45.000}}
}%
{\color[rgb]{0,0,0}\put(3511,-1141){\line( 0,-1){ 90}}
}%
{\color[rgb]{0,0,0}\put(3241,-1141){\line( 0,-1){ 90}}
}%
{\color[rgb]{0,0,0}\put(3691,-1681){\line( 0, 1){540}}
}%
{\color[rgb]{0,0,0}\put(3061,-1141){\framebox(720,270){}}
}%
{\color[rgb]{0,0,0}\put(2971,-871){\line( 0,-1){270}}
}%
{\color[rgb]{0,0,0}\put(3331,-1681){\line( 0,-1){ 90}}
}%
{\color[rgb]{0,0,0}\put(2971,-1681){\line( 0,-1){ 90}}
}%
{\color[rgb]{0,0,0}\put(3511,-1231){\line(-3,-2){270}}
\put(3241,-1411){\line( 0,-1){ 90}}
}%
{\color[rgb]{0,0,0}\put(3511,-1501){\line( 0,-1){180}}
}%
{\color[rgb]{0,0,0}\put(3241,-1231){\line(-3,-2){270}}
\put(2971,-1411){\line( 0,-1){ 90}}
}%
{\color[rgb]{0,0,0}\put(4681,-781){\line( 0, 1){ 90}}
\multiput(4681,-691)(-110.00000,0.00000){5}{\line(-1, 0){ 55.000}}
\multiput(4186,-691)(0.00000,-116.47059){9}{\line( 0,-1){ 58.235}}
\multiput(4186,-1681)(110.00000,0.00000){5}{\line( 1, 0){ 55.000}}
\put(4681,-1681){\line( 0, 1){ 90}}
}%
{\color[rgb]{0,0,0}\multiput(5041,-781)(0.00000,90.00000){2}{\line( 0, 1){ 45.000}}
\multiput(5041,-646)(-120.00000,0.00000){8}{\line(-1, 0){ 60.000}}
\multiput(4141,-646)(0.00000,-113.68421){10}{\line( 0,-1){ 56.842}}
\multiput(4141,-1726)(120.00000,0.00000){8}{\line( 1, 0){ 60.000}}
\multiput(5041,-1726)(0.00000,90.00000){2}{\line( 0, 1){ 45.000}}
}%
{\color[rgb]{0,0,0}\put(4321,-871){\line( 0,-1){450}}
}%
{\color[rgb]{0,0,0}\put(4591,-1501){\framebox(540,180){}}
}%
{\color[rgb]{0,0,0}\put(4681,-1141){\line( 0,-1){180}}
}%
{\color[rgb]{0,0,0}\put(5041,-1141){\line( 0,-1){180}}
}%
{\color[rgb]{0,0,0}\put(5041,-1501){\line( 0,-1){ 90}}
}%
{\color[rgb]{0,0,0}\put(4681,-1501){\line( 0,-1){ 90}}
}%
{\color[rgb]{0,0,0}\put(4501,-1141){\line( 0,-1){180}}
}%
{\color[rgb]{0,0,0}\put(5041,-871){\line( 0, 1){ 90}}
}%
{\color[rgb]{0,0,0}\put(4681,-871){\line( 0, 1){ 90}}
}%
{\color[rgb]{0,0,0}\put(4951,-1321){\line( 0, 1){180}}
}%
{\color[rgb]{0,0,0}\put(2881,-1681){\framebox(540,180){}}
}%
{\color[rgb]{0,0,0}\put(4411,-1141){\framebox(720,270){}}
}%
\put(4861,-1434){\makebox(0,0)[b]{\smash{{\SetFigFont{12}{14.4}{\rmdefault}{\mddefault}{\updefault}{\color[rgb]{0,0,0}$p_{k-1}$}%
}}}}
\put(4771,-1051){\makebox(0,0)[b]{\smash{{\SetFigFont{12}{14.4}{\rmdefault}{\mddefault}{\updefault}{\color[rgb]{0,0,0}$f$}%
}}}}
\put(3421,-1051){\makebox(0,0)[b]{\smash{{\SetFigFont{12}{14.4}{\rmdefault}{\mddefault}{\updefault}{\color[rgb]{0,0,0}$f$}%
}}}}
\put(3961,-1321){\makebox(0,0)[b]{\smash{{\SetFigFont{12}{14.4}{\rmdefault}{\mddefault}{\updefault}{\color[rgb]{0,0,0}$=$}%
}}}}
\put(3151,-1614){\makebox(0,0)[b]{\smash{{\SetFigFont{12}{14.4}{\rmdefault}{\mddefault}{\updefault}{\color[rgb]{0,0,0}$p_{k-1}$}%
}}}}
\put(5311,-1321){\makebox(0,0)[lb]{\smash{{\SetFigFont{12}{14.4}{\rmdefault}{\mddefault}{\updefault}{\color[rgb]{0,0,0}$=\Tr(f)$.}%
}}}}
\end{picture}%

%% file: term_a+b.latex
\setlength{\unitlength}{4144sp}%
\begingroup\makeatletter\ifx\SetFigFont\undefined%
\gdef\SetFigFont#1#2#3#4#5{%
  \reset@font\fontsize{#1}{#2pt}%
  \fontfamily{#3}\fontseries{#4}\fontshape{#5}%
  \selectfont}%
\fi\endgroup%
\begin{picture}(4344,2724)(1834,-3988)
{\color[rgb]{0,0,0}\thinlines
\put(3061,-2221){\circle*{10}}
}%
{\color[rgb]{0,0,0}\put(3106,-2221){\circle*{10}}
}%
{\color[rgb]{0,0,0}\put(3016,-2221){\circle*{10}}
}%
{\color[rgb]{0,0,0}\put(2521,-2221){\circle*{10}}
}%
{\color[rgb]{0,0,0}\put(2566,-2221){\circle*{10}}
}%
{\color[rgb]{0,0,0}\put(2476,-2221){\circle*{10}}
}%
{\color[rgb]{0,0,0}\put(2251,-2221){\circle*{10}}
}%
{\color[rgb]{0,0,0}\put(2296,-2221){\circle*{10}}
}%
{\color[rgb]{0,0,0}\put(2206,-2221){\circle*{10}}
}%
{\color[rgb]{0,0,0}\put(2251,-1906){\circle*{10}}
}%
{\color[rgb]{0,0,0}\put(2296,-1906){\circle*{10}}
}%
{\color[rgb]{0,0,0}\put(2206,-1906){\circle*{10}}
}%
{\color[rgb]{0,0,0}\put(2251,-1546){\circle*{10}}
}%
{\color[rgb]{0,0,0}\put(2296,-1546){\circle*{10}}
}%
{\color[rgb]{0,0,0}\put(2206,-1546){\circle*{10}}
}%
{\color[rgb]{0,0,0}\put(2656,-1726){\circle*{10}}
}%
{\color[rgb]{0,0,0}\put(2701,-1726){\circle*{10}}
}%
{\color[rgb]{0,0,0}\put(2611,-1726){\circle*{10}}
}%
{\color[rgb]{0,0,0}\put(3151,-1906){\circle*{10}}
}%
{\color[rgb]{0,0,0}\put(3196,-1906){\circle*{10}}
}%
{\color[rgb]{0,0,0}\put(3106,-1906){\circle*{10}}
}%
{\color[rgb]{0,0,0}\put(3286,-1546){\circle*{10}}
}%
{\color[rgb]{0,0,0}\put(3331,-1546){\circle*{10}}
}%
{\color[rgb]{0,0,0}\put(3241,-1546){\circle*{10}}
}%
{\color[rgb]{0,0,0}\put(2071,-2131){\framebox(1260,180){}}
}%
{\color[rgb]{0,0,0}\put(2071,-1861){\framebox(360,270){}}
}%
{\color[rgb]{0,0,0}\put(2971,-1861){\framebox(540,270){}}
}%
{\color[rgb]{0,0,0}\put(5446,-1906){\circle*{10}}
}%
{\color[rgb]{0,0,0}\put(5491,-1906){\circle*{10}}
}%
{\color[rgb]{0,0,0}\put(5401,-1906){\circle*{10}}
}%
{\color[rgb]{0,0,0}\put(5581,-1546){\circle*{10}}
}%
{\color[rgb]{0,0,0}\put(5626,-1546){\circle*{10}}
}%
{\color[rgb]{0,0,0}\put(5536,-1546){\circle*{10}}
}%
{\color[rgb]{0,0,0}\put(5266,-1861){\framebox(540,270){}}
}%
{\color[rgb]{0,0,0}\put(5311,-2221){\circle*{10}}
}%
{\color[rgb]{0,0,0}\put(5356,-2221){\circle*{10}}
}%
{\color[rgb]{0,0,0}\put(5266,-2221){\circle*{10}}
}%
{\color[rgb]{0,0,0}\put(4996,-1726){\circle*{10}}
}%
{\color[rgb]{0,0,0}\put(5041,-1726){\circle*{10}}
}%
{\color[rgb]{0,0,0}\put(4951,-1726){\circle*{10}}
}%
{\color[rgb]{0,0,0}\put(4681,-1546){\circle*{10}}
}%
{\color[rgb]{0,0,0}\put(4726,-1546){\circle*{10}}
}%
{\color[rgb]{0,0,0}\put(4636,-1546){\circle*{10}}
}%
{\color[rgb]{0,0,0}\put(4681,-1906){\circle*{10}}
}%
{\color[rgb]{0,0,0}\put(4726,-1906){\circle*{10}}
}%
{\color[rgb]{0,0,0}\put(4636,-1906){\circle*{10}}
}%
{\color[rgb]{0,0,0}\put(4996,-2221){\circle*{10}}
}%
{\color[rgb]{0,0,0}\put(5041,-2221){\circle*{10}}
}%
{\color[rgb]{0,0,0}\put(4951,-2221){\circle*{10}}
}%
{\color[rgb]{0,0,0}\put(4681,-2221){\circle*{10}}
}%
{\color[rgb]{0,0,0}\put(4726,-2221){\circle*{10}}
}%
{\color[rgb]{0,0,0}\put(4636,-2221){\circle*{10}}
}%
{\color[rgb]{0,0,0}\put(5446,-3346){\circle*{10}}
}%
{\color[rgb]{0,0,0}\put(5491,-3346){\circle*{10}}
}%
{\color[rgb]{0,0,0}\put(5401,-3346){\circle*{10}}
}%
{\color[rgb]{0,0,0}\put(5581,-2986){\circle*{10}}
}%
{\color[rgb]{0,0,0}\put(5626,-2986){\circle*{10}}
}%
{\color[rgb]{0,0,0}\put(5536,-2986){\circle*{10}}
}%
{\color[rgb]{0,0,0}\put(5266,-3301){\framebox(540,270){}}
}%
{\color[rgb]{0,0,0}\put(5311,-3661){\circle*{10}}
}%
{\color[rgb]{0,0,0}\put(5356,-3661){\circle*{10}}
}%
{\color[rgb]{0,0,0}\put(5266,-3661){\circle*{10}}
}%
{\color[rgb]{0,0,0}\put(4996,-3166){\circle*{10}}
}%
{\color[rgb]{0,0,0}\put(5041,-3166){\circle*{10}}
}%
{\color[rgb]{0,0,0}\put(4951,-3166){\circle*{10}}
}%
{\color[rgb]{0,0,0}\put(4681,-2986){\circle*{10}}
}%
{\color[rgb]{0,0,0}\put(4726,-2986){\circle*{10}}
}%
{\color[rgb]{0,0,0}\put(4636,-2986){\circle*{10}}
}%
{\color[rgb]{0,0,0}\put(4681,-3346){\circle*{10}}
}%
{\color[rgb]{0,0,0}\put(4726,-3346){\circle*{10}}
}%
{\color[rgb]{0,0,0}\put(4636,-3346){\circle*{10}}
}%
{\color[rgb]{0,0,0}\put(4996,-3661){\circle*{10}}
}%
{\color[rgb]{0,0,0}\put(5041,-3661){\circle*{10}}
}%
{\color[rgb]{0,0,0}\put(4951,-3661){\circle*{10}}
}%
{\color[rgb]{0,0,0}\put(4681,-3661){\circle*{10}}
}%
{\color[rgb]{0,0,0}\put(4726,-3661){\circle*{10}}
}%
{\color[rgb]{0,0,0}\put(4636,-3661){\circle*{10}}
}%
{\color[rgb]{0,0,0}\put(3331,-2131){\oval(180,180)[bl]}
\put(3331,-2131){\oval(180,180)[br]}
}%
{\color[rgb]{0,0,0}\put(2791,-2311){\oval(180,180)[tr]}
\put(2791,-2311){\oval(180,180)[tl]}
}%
{\color[rgb]{0,0,0}\put(5626,-2131){\oval(180,180)[bl]}
\put(5626,-2131){\oval(180,180)[br]}
}%
{\color[rgb]{0,0,0}\put(5266,-1501){\oval(180,180)[tr]}
\put(5266,-1501){\oval(180,180)[tl]}
}%
{\color[rgb]{0,0,0}\put(5266,-2941){\oval(180,180)[tr]}
\put(5266,-2941){\oval(180,180)[tl]}
}%
{\color[rgb]{0,0,0}\put(5806,-3301){\oval(180,180)[bl]}
\put(5806,-3301){\oval(180,180)[br]}
}%
{\color[rgb]{0,0,0}\put(3421,-1861){\line( 0,-1){270}}
}%
{\color[rgb]{0,0,0}\put(2611,-2131){\line( 0,-1){180}}
}%
{\color[rgb]{0,0,0}\put(3241,-1861){\line( 0,-1){ 90}}
}%
{\color[rgb]{0,0,0}\put(3421,-1591){\line( 0, 1){ 90}}
}%
{\color[rgb]{0,0,0}\put(3061,-1591){\line( 0, 1){ 90}}
}%
{\color[rgb]{0,0,0}\put(3151,-1591){\line( 0, 1){ 90}}
}%
{\color[rgb]{0,0,0}\put(2341,-1591){\line( 0, 1){ 90}}
}%
{\color[rgb]{0,0,0}\put(2881,-1951){\line( 0, 1){450}}
}%
{\color[rgb]{0,0,0}\put(2341,-1861){\line( 0,-1){ 90}}
}%
{\color[rgb]{0,0,0}\put(2161,-2131){\line( 0,-1){180}}
}%
{\color[rgb]{0,0,0}\put(2521,-1951){\line( 0, 1){450}}
}%
{\color[rgb]{0,0,0}\put(2161,-1591){\line( 0, 1){ 90}}
}%
{\color[rgb]{0,0,0}\put(2161,-1861){\line( 0,-1){ 90}}
}%
{\color[rgb]{0,0,0}\put(3061,-1951){\line( 0, 1){ 90}}
}%
{\color[rgb]{0,0,0}\multiput(3241,-2311)(-3.60000,7.20000){26}{\makebox(1.5875,11.1125){\tiny.}}
}%
{\color[rgb]{0,0,0}\multiput(2971,-2311)(-3.60000,7.20000){26}{\makebox(1.5875,11.1125){\tiny.}}
}%
{\color[rgb]{0,0,0}\put(2791,-1501){\line( 0,-1){450}}
}%
{\color[rgb]{0,0,0}\multiput(2791,-1501)(0.00000,90.00000){3}{\line( 0, 1){ 45.000}}
\multiput(2791,-1276)(-82.17391,0.00000){12}{\line(-1, 0){ 41.087}}
\multiput(1846,-1276)(0.00000,-86.89655){15}{\line( 0,-1){ 43.448}}
\multiput(1846,-2536)(90.00000,0.00000){9}{\line( 1, 0){ 45.000}}
\multiput(2611,-2536)(0.00000,90.00000){3}{\line( 0, 1){ 45.000}}
}%
{\color[rgb]{0,0,0}\multiput(2521,-1501)(0.00000,72.00000){3}{\line( 0, 1){ 36.000}}
\multiput(2521,-1321)(-84.00000,0.00000){8}{\line(-1, 0){ 42.000}}
\multiput(1891,-1321)(0.00000,-86.66667){14}{\line( 0,-1){ 43.333}}
\multiput(1891,-2491)(83.07692,0.00000){7}{\line( 1, 0){ 41.538}}
\multiput(2431,-2491)(0.00000,72.00000){3}{\line( 0, 1){ 36.000}}
}%
{\color[rgb]{0,0,0}\multiput(2341,-1501)(0.00000,90.00000){2}{\line( 0, 1){ 45.000}}
\multiput(2341,-1366)(-90.00000,0.00000){5}{\line(-1, 0){ 45.000}}
\multiput(1936,-1366)(0.00000,-86.40000){13}{\line( 0,-1){ 43.200}}
\multiput(1936,-2446)(90.00000,0.00000){5}{\line( 1, 0){ 45.000}}
\multiput(2341,-2446)(0.00000,90.00000){2}{\line( 0, 1){ 45.000}}
}%
{\color[rgb]{0,0,0}\multiput(2161,-1501)(0.00000,60.00000){2}{\line( 0, 1){ 30.000}}
\multiput(2161,-1411)(-72.00000,0.00000){3}{\line(-1, 0){ 36.000}}
\multiput(1981,-1411)(0.00000,-86.08696){12}{\line( 0,-1){ 43.043}}
\multiput(1981,-2401)(72.00000,0.00000){3}{\line( 1, 0){ 36.000}}
\multiput(2161,-2401)(0.00000,60.00000){2}{\line( 0, 1){ 30.000}}
}%
{\color[rgb]{0,0,0}\multiput(2881,-1501)(0.00000,90.00000){3}{\line( 0, 1){ 45.000}}
\multiput(2881,-1276)(90.00000,0.00000){10}{\line( 1, 0){ 45.000}}
\multiput(3736,-1276)(0.00000,-86.89655){15}{\line( 0,-1){ 43.448}}
\multiput(3736,-2536)(-82.80000,0.00000){13}{\line(-1, 0){ 41.400}}
\multiput(2701,-2536)(0.00000,90.00000){3}{\line( 0, 1){ 45.000}}
}%
{\color[rgb]{0,0,0}\multiput(3061,-1501)(0.00000,72.00000){3}{\line( 0, 1){ 36.000}}
\multiput(3061,-1321)(84.00000,0.00000){8}{\line( 1, 0){ 42.000}}
\multiput(3691,-1321)(0.00000,-86.66667){14}{\line( 0,-1){ 43.333}}
\multiput(3691,-2491)(-85.26316,0.00000){10}{\line(-1, 0){ 42.632}}
\multiput(2881,-2491)(0.00000,72.00000){3}{\line( 0, 1){ 36.000}}
}%
{\color[rgb]{0,0,0}\multiput(3151,-1501)(0.00000,90.00000){2}{\line( 0, 1){ 45.000}}
\multiput(3151,-1366)(90.00000,0.00000){6}{\line( 1, 0){ 45.000}}
\multiput(3646,-1366)(0.00000,-86.40000){13}{\line( 0,-1){ 43.200}}
\multiput(3646,-2446)(-90.00000,0.00000){8}{\line(-1, 0){ 45.000}}
\multiput(2971,-2446)(0.00000,90.00000){2}{\line( 0, 1){ 45.000}}
}%
{\color[rgb]{0,0,0}\multiput(3421,-1501)(0.00000,60.00000){2}{\line( 0, 1){ 30.000}}
\multiput(3421,-1411)(72.00000,0.00000){3}{\line( 1, 0){ 36.000}}
\multiput(3601,-1411)(0.00000,-86.08696){12}{\line( 0,-1){ 43.043}}
\multiput(3601,-2401)(-80.00000,0.00000){5}{\line(-1, 0){ 40.000}}
\multiput(3241,-2401)(0.00000,60.00000){2}{\line( 0, 1){ 30.000}}
}%
{\color[rgb]{0,0,0}\put(2431,-2311){\line( 0, 1){180}}
}%
{\color[rgb]{0,0,0}\put(2341,-2311){\line( 0, 1){180}}
}%
{\color[rgb]{0,0,0}\put(5716,-1861){\line( 0,-1){270}}
}%
{\color[rgb]{0,0,0}\put(5536,-1861){\line( 0,-1){ 90}}
}%
{\color[rgb]{0,0,0}\put(5716,-1591){\line( 0, 1){ 90}}
}%
{\color[rgb]{0,0,0}\put(5356,-1591){\line( 0, 1){ 90}}
}%
{\color[rgb]{0,0,0}\put(5446,-1591){\line( 0, 1){ 90}}
}%
{\color[rgb]{0,0,0}\put(5176,-1951){\line( 0, 1){450}}
}%
{\color[rgb]{0,0,0}\put(5356,-1951){\line( 0, 1){ 90}}
}%
{\color[rgb]{0,0,0}\put(5446,-2311){\line( 0, 1){180}}
}%
{\color[rgb]{0,0,0}\put(5176,-2311){\line( 0, 1){180}}
}%
{\color[rgb]{0,0,0}\put(5086,-1501){\line( 0,-1){450}}
}%
{\color[rgb]{0,0,0}\multiput(5446,-1501)(0.00000,90.00000){2}{\line( 0, 1){ 45.000}}
\multiput(5446,-1366)(90.00000,0.00000){6}{\line( 1, 0){ 45.000}}
\multiput(5941,-1366)(0.00000,-86.40000){13}{\line( 0,-1){ 43.200}}
\multiput(5941,-2446)(-90.00000,0.00000){9}{\line(-1, 0){ 45.000}}
\multiput(5176,-2446)(0.00000,90.00000){2}{\line( 0, 1){ 45.000}}
}%
{\color[rgb]{0,0,0}\multiput(5716,-1501)(0.00000,60.00000){2}{\line( 0, 1){ 30.000}}
\multiput(5716,-1411)(72.00000,0.00000){3}{\line( 1, 0){ 36.000}}
\multiput(5896,-1411)(0.00000,-86.08696){12}{\line( 0,-1){ 43.043}}
\multiput(5896,-2401)(-81.81818,0.00000){6}{\line(-1, 0){ 40.909}}
\multiput(5446,-2401)(0.00000,60.00000){2}{\line( 0, 1){ 30.000}}
}%
{\color[rgb]{0,0,0}\put(4906,-1951){\line( 0, 1){450}}
}%
{\color[rgb]{0,0,0}\put(4771,-1861){\line( 0,-1){ 90}}
}%
{\color[rgb]{0,0,0}\put(4771,-1591){\line( 0, 1){ 90}}
}%
{\color[rgb]{0,0,0}\put(4591,-1591){\line( 0, 1){ 90}}
}%
{\color[rgb]{0,0,0}\put(4591,-1861){\line( 0,-1){ 90}}
}%
{\color[rgb]{0,0,0}\put(4546,-1861){\framebox(270,270){}}
}%
{\color[rgb]{0,0,0}\put(4546,-2131){\framebox(1080,180){}}
}%
{\color[rgb]{0,0,0}\put(5086,-2131){\line( 0,-1){180}}
}%
{\color[rgb]{0,0,0}\put(4591,-2131){\line( 0,-1){180}}
}%
{\color[rgb]{0,0,0}\multiput(4771,-1501)(0.00000,90.00000){2}{\line( 0, 1){ 45.000}}
\multiput(4771,-1366)(-90.00000,0.00000){5}{\line(-1, 0){ 45.000}}
\multiput(4366,-1366)(0.00000,-86.40000){13}{\line( 0,-1){ 43.200}}
\multiput(4366,-2446)(90.00000,0.00000){5}{\line( 1, 0){ 45.000}}
\multiput(4771,-2446)(0.00000,90.00000){2}{\line( 0, 1){ 45.000}}
}%
{\color[rgb]{0,0,0}\multiput(4591,-1501)(0.00000,60.00000){2}{\line( 0, 1){ 30.000}}
\multiput(4591,-1411)(-72.00000,0.00000){3}{\line(-1, 0){ 36.000}}
\multiput(4411,-1411)(0.00000,-86.08696){12}{\line( 0,-1){ 43.043}}
\multiput(4411,-2401)(72.00000,0.00000){3}{\line( 1, 0){ 36.000}}
\multiput(4591,-2401)(0.00000,60.00000){2}{\line( 0, 1){ 30.000}}
}%
{\color[rgb]{0,0,0}\multiput(5086,-1501)(0.00000,90.00000){3}{\line( 0, 1){ 45.000}}
\multiput(5086,-1276)(-85.26316,0.00000){10}{\line(-1, 0){ 42.632}}
\multiput(4276,-1276)(0.00000,-86.89655){15}{\line( 0,-1){ 43.448}}
\multiput(4276,-2536)(85.26316,0.00000){10}{\line( 1, 0){ 42.632}}
\multiput(5086,-2536)(0.00000,72.00000){3}{\line( 0, 1){ 36.000}}
}%
{\color[rgb]{0,0,0}\put(4771,-2311){\line( 0, 1){180}}
}%
{\color[rgb]{0,0,0}\multiput(4906,-1501)(0.00000,72.00000){3}{\line( 0, 1){ 36.000}}
\multiput(4906,-1321)(-90.00000,0.00000){7}{\line(-1, 0){ 45.000}}
\multiput(4321,-1321)(0.00000,-86.66667){14}{\line( 0,-1){ 43.333}}
\multiput(4321,-2491)(90.00000,0.00000){7}{\line( 1, 0){ 45.000}}
\multiput(4906,-2491)(0.00000,72.00000){3}{\line( 0, 1){ 36.000}}
}%
{\color[rgb]{0,0,0}\put(4906,-2311){\line( 0, 1){180}}
}%
{\color[rgb]{0,0,0}\put(5536,-3301){\line( 0,-1){ 90}}
}%
{\color[rgb]{0,0,0}\put(5716,-3031){\line( 0, 1){ 90}}
}%
{\color[rgb]{0,0,0}\put(5356,-3031){\line( 0, 1){ 90}}
}%
{\color[rgb]{0,0,0}\put(5446,-3031){\line( 0, 1){ 90}}
}%
{\color[rgb]{0,0,0}\put(5176,-3391){\line( 0, 1){450}}
}%
{\color[rgb]{0,0,0}\put(5356,-3391){\line( 0, 1){ 90}}
}%
{\color[rgb]{0,0,0}\put(5446,-3751){\line( 0, 1){180}}
}%
{\color[rgb]{0,0,0}\put(5176,-3751){\line( 0, 1){180}}
}%
{\color[rgb]{0,0,0}\put(5086,-2941){\line( 0,-1){450}}
}%
{\color[rgb]{0,0,0}\multiput(5446,-2941)(0.00000,90.00000){2}{\line( 0, 1){ 45.000}}
\multiput(5446,-2806)(84.70588,0.00000){9}{\line( 1, 0){ 42.353}}
\multiput(6166,-2806)(0.00000,-86.40000){13}{\line( 0,-1){ 43.200}}
\multiput(6166,-3886)(-86.08696,0.00000){12}{\line(-1, 0){ 43.043}}
\multiput(5176,-3886)(0.00000,90.00000){2}{\line( 0, 1){ 45.000}}
}%
{\color[rgb]{0,0,0}\multiput(5716,-2941)(0.00000,60.00000){2}{\line( 0, 1){ 30.000}}
\multiput(5716,-2851)(90.00000,0.00000){5}{\line( 1, 0){ 45.000}}
\multiput(6121,-2851)(0.00000,-86.08696){12}{\line( 0,-1){ 43.043}}
\multiput(6121,-3841)(-90.00000,0.00000){8}{\line(-1, 0){ 45.000}}
\multiput(5446,-3841)(0.00000,60.00000){2}{\line( 0, 1){ 30.000}}
}%
{\color[rgb]{0,0,0}\put(4906,-3391){\line( 0, 1){450}}
}%
{\color[rgb]{0,0,0}\put(4771,-3301){\line( 0,-1){ 90}}
}%
{\color[rgb]{0,0,0}\put(4771,-3031){\line( 0, 1){ 90}}
}%
{\color[rgb]{0,0,0}\put(4591,-3031){\line( 0, 1){ 90}}
}%
{\color[rgb]{0,0,0}\put(4591,-3301){\line( 0,-1){ 90}}
}%
{\color[rgb]{0,0,0}\put(4546,-3301){\framebox(270,270){}}
}%
{\color[rgb]{0,0,0}\put(4546,-3571){\framebox(1080,180){}}
}%
{\color[rgb]{0,0,0}\put(5086,-3571){\line( 0,-1){180}}
}%
{\color[rgb]{0,0,0}\put(4591,-3571){\line( 0,-1){180}}
}%
{\color[rgb]{0,0,0}\multiput(4771,-2941)(0.00000,90.00000){2}{\line( 0, 1){ 45.000}}
\multiput(4771,-2806)(-90.00000,0.00000){5}{\line(-1, 0){ 45.000}}
\multiput(4366,-2806)(0.00000,-86.40000){13}{\line( 0,-1){ 43.200}}
\multiput(4366,-3886)(90.00000,0.00000){5}{\line( 1, 0){ 45.000}}
\multiput(4771,-3886)(0.00000,90.00000){2}{\line( 0, 1){ 45.000}}
}%
{\color[rgb]{0,0,0}\multiput(4591,-2941)(0.00000,60.00000){2}{\line( 0, 1){ 30.000}}
\multiput(4591,-2851)(-72.00000,0.00000){3}{\line(-1, 0){ 36.000}}
\multiput(4411,-2851)(0.00000,-86.08696){12}{\line( 0,-1){ 43.043}}
\multiput(4411,-3841)(72.00000,0.00000){3}{\line( 1, 0){ 36.000}}
\multiput(4591,-3841)(0.00000,60.00000){2}{\line( 0, 1){ 30.000}}
}%
{\color[rgb]{0,0,0}\multiput(5086,-2941)(0.00000,90.00000){3}{\line( 0, 1){ 45.000}}
\multiput(5086,-2716)(-85.26316,0.00000){10}{\line(-1, 0){ 42.632}}
\multiput(4276,-2716)(0.00000,-86.89655){15}{\line( 0,-1){ 43.448}}
\multiput(4276,-3976)(85.26316,0.00000){10}{\line( 1, 0){ 42.632}}
\multiput(5086,-3976)(0.00000,72.00000){3}{\line( 0, 1){ 36.000}}
}%
{\color[rgb]{0,0,0}\put(4771,-3751){\line( 0, 1){180}}
}%
{\color[rgb]{0,0,0}\multiput(4906,-2941)(0.00000,72.00000){3}{\line( 0, 1){ 36.000}}
\multiput(4906,-2761)(-90.00000,0.00000){7}{\line(-1, 0){ 45.000}}
\multiput(4321,-2761)(0.00000,-86.66667){14}{\line( 0,-1){ 43.333}}
\multiput(4321,-3931)(90.00000,0.00000){7}{\line( 1, 0){ 45.000}}
\multiput(4906,-3931)(0.00000,72.00000){3}{\line( 0, 1){ 36.000}}
}%
{\color[rgb]{0,0,0}\put(4906,-3751){\line( 0, 1){180}}
}%
{\color[rgb]{0,0,0}\put(5536,-3571){\line( 0,-1){180}}
\put(5536,-3751){\line( 1, 0){495}}
\put(6031,-3751){\line( 0, 1){810}}
\put(6031,-2941){\line(-1, 0){135}}
\put(5896,-2941){\line( 0,-1){360}}
}%
\put(2701,-2064){\makebox(0,0)[b]{\smash{{\SetFigFont{12}{14.4}{\rmdefault}{\mddefault}{\updefault}{\color[rgb]{0,0,0}$p_{a+b+c-1}$}%
}}}}
\put(2251,-1771){\makebox(0,0)[b]{\smash{{\SetFigFont{12}{14.4}{\rmdefault}{\mddefault}{\updefault}{\color[rgb]{0,0,0}$f_{(1)}$}%
}}}}
\put(3241,-1771){\makebox(0,0)[b]{\smash{{\SetFigFont{12}{14.4}{\rmdefault}{\mddefault}{\updefault}{\color[rgb]{0,0,0}$f_{(2)}$}%
}}}}
\put(5536,-1771){\makebox(0,0)[b]{\smash{{\SetFigFont{12}{14.4}{\rmdefault}{\mddefault}{\updefault}{\color[rgb]{0,0,0}$f_{(2)}$}%
}}}}
\put(5536,-3211){\makebox(0,0)[b]{\smash{{\SetFigFont{12}{14.4}{\rmdefault}{\mddefault}{\updefault}{\color[rgb]{0,0,0}$f_{(2)}$}%
}}}}
\put(4996,-2064){\makebox(0,0)[b]{\smash{{\SetFigFont{12}{14.4}{\rmdefault}{\mddefault}{\updefault}{\color[rgb]{0,0,0}$p_{a+b+c-1}$}%
}}}}
\put(4681,-1771){\makebox(0,0)[b]{\smash{{\SetFigFont{12}{14.4}{\rmdefault}{\mddefault}{\updefault}{\color[rgb]{0,0,0}$f_{(1)}$}%
}}}}
\put(4996,-3504){\makebox(0,0)[b]{\smash{{\SetFigFont{12}{14.4}{\rmdefault}{\mddefault}{\updefault}{\color[rgb]{0,0,0}$p_{a+b+c-1}$}%
}}}}
\put(4681,-3211){\makebox(0,0)[b]{\smash{{\SetFigFont{12}{14.4}{\rmdefault}{\mddefault}{\updefault}{\color[rgb]{0,0,0}$f_{(1)}$}%
}}}}
\put(4006,-1906){\makebox(0,0)[b]{\smash{{\SetFigFont{12}{14.4}{\rmdefault}{\mddefault}{\updefault}{\color[rgb]{0,0,0}$=$}%
}}}}
\put(4006,-3346){\makebox(0,0)[b]{\smash{{\SetFigFont{12}{14.4}{\rmdefault}{\mddefault}{\updefault}{\color[rgb]{0,0,0}$=$}%
}}}}
\end{picture}%